\documentclass[journal]{IEEEtran}
\usepackage[english]{babel}  		 			 
\usepackage{hyperref}	 			 
\usepackage[utf8]{inputenc}
\usepackage[T1]{fontenc}       %
\usepackage{amsmath,graphicx,amsthm,pstricks}
\usepackage{amssymb} 
\usepackage{multirow}
\usepackage{graphicx}
\usepackage{epstopdf}
\usepackage{xcolor}
\usepackage{multirow}
\usepackage{algorithm}
\usepackage{algorithmic}
\usepackage{slashbox}
\usepackage[version-1-compatibility]{siunitx}
\usepackage{xspace}
\newcommand{\lp}[1]{\ensuremath{\ell_{#1}}\xspace} 
\newcommand{\lz}{\ensuremath{\ell_0}\xspace} 
\newcommand{\lo}{\ensuremath{\ell_1}\xspace} 
\newcommand{\lzx}[1]{\ensuremath{\ell_0(#1)}\xspace} 
\newcommand{\lpx}[2]{\ensuremath{\ell_{#1}(#2)}\xspace} 
\newcommand{\lpxp}[2]{\ensuremath{\ell_{#1}^{#1}(#2)}\xspace} 
\newcommand{\spoq}{\ensuremath{\text{SPOQ}}\xspace} 

\newcommand{\celo}{\ensuremath{\mathrm{CEL}0}\xspace} 

\newcommand\allx{}
\newcommand\allxx{}

\usepackage[final]{changes}
\definechangesauthor[name={Afef}, color=black]{af}
\definechangesauthor[name={Emilie}, color=black]{ec}
\definechangesauthor[name={Jean-Christophe}, color=black]{jp}
\definechangesauthor[name={Laurent}, color=black]{ld}
\newcommand{\lda}[1]{#1}
\newcommand{\ldr}[2]{#1}
\newcommand{\ldd}[2]{#2}

\newcommand{\afa}[1]{#1}
\newcommand{\afr}[2]{#1}
\newcommand{\afd}[2]{#2}

\newcommand{\eca}[1]{#1}

\definechangesauthor[name={Afef}, color=blue]{raf}
\definechangesauthor[name={Emilie}, color=orange]{rec}
\definechangesauthor[name={Jean-Christophe}, color=red]{rjp}
\definechangesauthor[name={Laurent}, color=green]{rld}
\newcommand{\rlda}[1]{\added[id=rld]{#1}}

\newcommand{\cred}{\textcolor{black}}
\newcommand{\minimize}[2]{\ensuremath{\underset{\substack{{#1}}}%
{\mathrm{minimize}}\;\;#2 }}
\newcommand{\argmin}[2]{\ensuremath{\underset{\substack{{#1}}}%
{\mathrm{argmin}}\;\;#2 }}
\newcommand\ex{\ensuremath{\mathsf{\textbf{x}}}}
\newcommand\ey{\ensuremath{\mathsf{\textbf{y}}}}
\newcommand\ez{\ensuremath{\mathsf{\textbf{z}}}}

\newcommand\er{\ensuremath{\mathsf{\textbf{r}}}}
\newcommand\eb{\ensuremath{\mathsf{\textbf{b}}}}
\newcommand\ed{\ensuremath{\mathsf{\textbf{d}}}}
\newcommand\eI{\ensuremath{\pmb{I}}}
\newcommand\eD{\ensuremath{\pmb{D}}}
\newcommand\BB{\ensuremath{\mathcal{B}}}

\newcommand\eA{\ensuremath{\pmb{A}}}
\newcommand\zerob{\ensuremath{\pmb{0}}}

\newcommand\eR{\ensuremath{\mathbb{R}}}
\newcommand\eN{\ensuremath{\mathbb{N}}}
\newcommand\eC{\ensuremath{\mathcal{C}}}
\newcommand{\sign}{\ensuremath{\operatorname{sign}}} 
\newcommand\DD{\ensuremath{\operatorname{Diag}}}

\newcommand{\prox}{\ensuremath{\operatorname{prox}}}

{
\theoremstyle{definition}
\newtheorem{theorem}{Theorem}
\newtheorem{lemma}{Lemma}
\newtheorem{proposition}{Proposition}

\newtheorem{remark}{Remark}

\newtheorem{assumption}{Assumption} 
}

\graphicspath{{Plots/}}
\usepackage{epstopdf}

\begin{document}
\title{{SPOQ}  $\ell_p$-over-$\ell_q$  Regularization for Sparse Signal Recovery applied to Mass Spectrometry}
\author{Afef Cherni,~\IEEEmembership{Member,~IEEE,}
        Emilie Chouzenoux,~\IEEEmembership{Senior member,~IEEE,}
        Laurent Duval,~\IEEEmembership{Member,~IEEE,}      
        and~Jean-Christophe Pesquet,~\IEEEmembership{Fellow,~IEEE}
\thanks{A. Cherni is with the Department
of Aix-Marseille Univ, CNRS, Centrale Marseille, I2M, Marseille, France.}
\thanks{E. Chouzenoux and J.-C. Pesquet are with Universit\'e Paris-Saclay, CentraleSup\'elec, Inria, Centre de Vision Num\'erique, 91190 Gif-sur-Yvette, France.}
\thanks{L. Duval is with ESIEE Paris, University Gustave Eiffel, LIGM, Noisy-le-Grand, and IFP Energies nouvelles, Rueil-Malmaison, France.}
\thanks{This work was partly supported by the Excellence Initiative of Aix-Marseille University --- A*Midex, a French "Investissements d'Avenir" program, the Agence Nationale de la Recherche of France under MAJIC (ANR-17-CE40-0004-01) project, and the Institut Universitaire de France}
}
\markboth{Journal of \LaTeX\ Class Files,~Vol.~XX, No.~XX, Month~XXXX}%
{Shell \MakeLowercase{\textit{et al.}}: Bare Demo of IEEEtran.cls for IEEE Journals}

\maketitle

\begin{abstract}
Underdetermined or ill-posed inverse problems require additional information for \ldd{d} sound solutions with tractable optimization algorithms. Sparsity yields consequent  heuristics to that matter, with numerous applications  in signal restoration,  image recovery, or machine learning. Since the $\ell_0$ count measure is barely tractable, many statistical or learning approaches have invested in computable proxies, such as the $\ell_1$ norm. However, the latter does  not exhibit  the desirable property of  scale invariance for sparse data. \cred{Extending} the SOOT Euclidean/Taxicab \cred{$\ell_1$-over-$\ell_2$} norm-ratio initially introduced for blind deconvolution, we propose SPOQ, a family of smoothed \cred{ (approximately)} scale-invariant penalty functions. It consists of a Lipschitz-differentiable surrogate for $\ell_p$-over-$\ell_q$ quasi-norm/norm ratios with $p\in\,]0,2[$ and $q\ge 2$. This surrogate is embedded into a novel majorize-minimize trust-region approach, generalizing the variable metric forward-backward algorithm. For naturally sparse mass-spectrometry  signals, we show that SPOQ significantly outperforms  $\ell_0$, $\ell_1$, Cauchy, Welsch, \cred{SCAD} and \celo penalties on several performance measures. 
Guidelines on SPOQ hyperparameters tuning are also provided, suggesting simple data-driven choices.
\end{abstract}

\begin{IEEEkeywords}
Inverse problems, majorize-minimize method, mass spectrometry, nonconvex optimization, nonsmooth optimization, norm ratio, quasinorm, sparsity.
\end{IEEEkeywords}
 
%
\IEEEpeerreviewmaketitle



\section{Introduction and background\label{Introduction}}

\subsection{\cred{Measuring and optimizing sparsity in} data science}
\cred{Sparsity classically refers to a heuristic principle \cite{Laird_J_1919_j-monist_law_p,Bontly_T_2005_j-mind-lang_modified_orppawm}: simplest representations or models may be better (for representation, interpretation, inference, etc.)}\footnote{\cred{The law of parsimony is used  in theology, history, social and empirical disciplines; Occam's razor after  William of Ockham:  ``Entities should not be multiplied without necessity" }.}.
%
%
\cred{In data sciences, it can qualify} degrees of freedom for parametric models,  reduce a search space,  define stopping criteria,  bound  filter support,  simplify signals or images with  meaningful structures. \cred{Sparse signal recovery amounts to restoring data from inherently sparse information (e.g., spiking neurons, analytical chemistry) degraded by smoothing kernels and noise}. From partial observations, it  \cred{selects one solution}, among all potential \cred{consistent model candidates}. \cred{We review here precise formulations for sparsity measures, especially used in deconvolution, restoration or recovery. Starting with norm and quasinorm based regularizations, we will then address penalties made of ratios of those, that were recently proposed to enforce (near) scale-invariance and that will be at the core of our paper.}

The natural playground for sparsity in discrete time series analysis is the space   of  almost-zero  real sequences:  \cred{$c_{00}(\eR)  = \{ \ex=(x_n)_{n\in \eN} \;| \; \exists N \in \eN$ such that $ x_n=0\rlda{,}\;  \forall n \geq N\}$}. It is closed under finite addition and convolution \cite[p. 597]{Bronshtein_I_2007_book_handbook_m},  \cite[Chapter 2]{Albiac_F_2006_book_topic_bst}. Unless stated otherwise, in the following, we  consider sparse finite sequences  (hence, in $c_{00}(\eR)$), each being associated with a vector $\ex=(x_n)_{1 \leq n \leq N} \in \eR^N$. 
The cornerstone measure of parsimony is 
the count index,
i.e. the number of non-zero terms in $\ex$,  denoted by  $\lzx{\ex}$. 
It is also called \emph{cardinality function},  \emph{numerosity measure},  \emph{parsimony}, or \emph{sparsity}. For $p\in\,]0,+\infty[$, we define 
\cred{$\lpxp{p}{\ex} =  \sum_{n=1}^N|x_n|^p $}, which is a norm for $p \ge 1$. 
For quasinorms $\ell_p$, $p <1$, a  \cred{weaker triangle inequality holds:} $\| \ex+\ey \| \le K \left(\| \ex\|+\| \ey\| \right)$ with\footnote{The lowest $K$, modulus of concavity of the quasinorm,  saturates to $1$ for norms. For $0<p\le1$,  $\ell_p(\ex+\ey) \le 2^{\frac{1-p}{p}}  \big(\ell_p(\ex)+\ell_p(\ey)\big)$.}  $K\in\, [1,+\infty[$.
The \lp{p} quasinorm ($p<1$) is sometimes called $p$-norm-like   \cite{Rao_B_1999_j-ieee-tsp_affine_smbbs}.
\lz is  piecewise constant, nonsmooth and nonconvex.  
It is often considered as unusable for data optimization  in large linear systems,\footnote{Other denominations are \emph{subset selection}, \emph{minimum weight solution}, \emph{sparse null-space}, or \emph{minimum set cover}.}
since it leads to NP-hard  problems \cite{Natarajan_B_1995_j-siam-j-comput_sparse_asls,Fortnow_L_2009_j-comm-acm_status_pvnpp}. Under  drastic conditions, an $\lz$-problem  can \ldd{however} be solved exactly using a convex relation by surrogate penalties, like the \lo-norm \cite{Candes_E_2006_j-comm-pure-appl-math_stable_sriim}. 
\cred{Such conditions are rarely met in practice. The } use of the  \lo-norm yields approximate solutions and becomes more heuristic \cite{Ramirez_C_2013_j-uncertain-syst_why_l1gal0ge}.
Mixed-integer programming  reformulations using  branch-and-bound methods  \cite{BenMhenni_R_2019_PREPRINT_global_osslsp} are possible, \cred{yet often restricted to} relatively small-sized problems.

Norm- or quasinorm-based penalties have subsequently played an important role in sparse data processing or parsimonious modeling for \cred{high-dimensional} regression. Squared Euclidean norm   $\ell_2^2$  possesses efficient implementations  but often heavily degrades data sharpness. As a data fidelity term, the $\ell_2^2$ cost function alone cannot address, at the same time,  residual noise properties and additional data assumptions. 
It can be supplemented by various  variational regularizations, \emph{\`a la Tikhonov} \cite{Tikhonov_A_1963_j-dan-sssr_sol_ipmr}. Those act on the composition of  data with a well-suited sparsifying operator, e.g.  identity, gradients, higher-order derivatives, or wavelet frames \cite{Jacques_L_2011_j-sp_panorama_mgrisdfs}. The $\ell_2^2$  penalty case corresponds to ridge regression \cite{Hoerl_A_1970_j-technometrics_ridge_ranp}. In  basis pursuit \cite{Chen_S_1998_j-siam-j-sci-comput_atomic_dbp}, feature selection \cite{Bradley_P_1998_p-icml_feature_scmsvm}, or \emph{lasso} method (least absolute shrinkage and selection operator \cite{Tibshirani_R_1996_j-r-stat-soc-b-stat-methodol_regression_sslasso}), the one-norm  \lo or taxicab distance is  preferred, as it promotes a form of sparsity\ldd{ on the solution}. Solutions to the "$\ell_2^2$ fidelity plus $\ell_1$" regularization problem are related to total variation regularization in  image restoration \cite{Fu_H_2006_j-siam-j-sci-comput_efficient_mmml2l1l1l1nir}. \cred{It can be combined with higher-order derivatives for trend filtering and source separation in analytical chemistry} \cite{Ning_X_2014_j-chemometr-intell-lab-syst_chromatogram_bedusbeads}.
A convex combination of ridge and  lasso  regularizations yields  the elastic net regularization \cite{Zou_H_2005_j-r-stat-soc-b-stat-methodol_regularization_vsen,DeMol_C_2009_j-complexity_elastice-net_rlt}.
Other convex $p$-norms ($p \ge 1$) regularizations have been addressed, as in  bridge regression which interpolates between lasso and ridge \cite{Fu_W_1998_j-comput-graph-stat_penalized_rbvl}.
Expecting sparser solutions in practice, non-convex  least-squares plus $\ell_p$ ($p < 1$) problems have been addressed   \cite{Foucart_S_2009_j-acha_j-acha_sparsest_sulsvlqm0q1}. \cred{Other non-convex metrics have also been used in sparse regression and variable selection \cite{Soubies_E_2016_PREPRINT_unified_ecpl2l0m,Wen_F_2018_j-ieee-access_survey_nrbslrrspsml}}.
Although \emph{a priori} appealing for sparse data restoration, such problems retain NP-hard complexities \cite{Chen_X_2012_j-math-programm_complexity_ul2lpm}.
Another caveat of the above  norm/quasinorm penalties, as  proxies for reasonable  approximations to \lz, is their scale-variance:  norms and quasinorms satisfy the absolute homogeneity axiom ($\lpx{p}{\lambda\ex}=|\lambda|\lpx{p}{\ex}$, for $\lambda\in \eR $). 

\subsection{Penalties with quasinorm and norm ratios}
\cred{Better scale-invariance can be obtained by a proper normalization of quasinorms or norms penalties.}
We thus investigate a broader family of (quasi-)norm  ratios \cred{$\ell_p$-over-$\ell_q$}  with couples $(p,q)\in ]0,2[ \times [2,\infty[$, based on both their counting properties and probabilistic interpretation. First,  we have   equivalence relations in finite dimension:
\begin{equation}
\lpx{q}{\ex} \leq \lpx{p}{\ex}  \leq \lzx{\ex}^{\frac{1}{p}-\frac{1}{q}}\lpx{q}{\ex} 
\label{eq_lp-lq}
\end{equation}
with  $ p \leq q$, from  the standard power-mean inequality \cite{Beckenbach_E_1946_j-am-math-mon_inequality_j} implying classical $\ell_p$-space embeddings and  generalized  Rogers-H\"older's inequalities.
The LHS in \eqref{eq_lp-lq} is attained when $\ex$  \cred{is a sparsest signal in $c_{00}(\eR)$} 
 with \cred{a single}  non-zero component. The RHS is reached by a maximally non-sparse $\ex$, with \cred{all  samples  set} to a non-zero constant \cite{Hurley_N_2009_j-ieee-tit_comparing_ms}.
Thus,  \cred{$\ell_p$-over-$\ell_q$} quasinorm-ratios provide  interesting proxies  for a sparseness measure of $\ex$, to  quantify how much the ``action'' or ``energy'' of a discrete signal  is concentrated into only a few of its components \cred{ (\emph{cf.} Figure \ref{Fig:Testsparsitydegree})}.
They are invariant under  integer (circular) shift or sample shuffling in the sequence, and under  non-zero scale change (or $0$-degree homogeneity).
Those ratios are sometimes termed $pq$-means. 

For every $p \in ]0,2[$ and $q \in [2, +\infty[$, we thus define:
\begin{equation}
\left(\ell_p / \ell_q (\ex) \right)^{p} = \sum_{n =1}^N \left(\frac{|x_n|^q}{\sum_{n' =1}^N |x_{n'}|^q} \right)^{p/q} \,.
\label{eqlplqexact}
\end{equation}
\cred{Interpreting} the inner term 
\begin{equation}
p_n = \frac{|x_n|^q}{\sum_{n' =1}^N |x_{n'}|^q}
\end{equation}
 as a \cred{unitless} discrete probability distribution,  \cred{$\ell_p$-over-$\ell_q$ quasinorm-ratio} rewrites as an increasing function ($u \to u^{1/p}$) of a sum of concave functions ($u \to u^{p/q}$ when $ p \leq q$) of probabilities. The minimization of such an additive  information  cost  function 
\cite{Coifman_R_1992_j-ieee-tit_entropy-bsed_abbs,Sikic_H_2001_j-acha_information_cf}, 
a special case of Schur-concave functionals \cite{Aczel_J_1975_book_measures_ic,Marshall_A_1979_book_inequalities_tma}, is used for instance in best basis selection \cite{Leporini_D_1999_incoll_best_brbpsmbiwbm}.
Thus, special cases of  \cred{$\ell_p$-over-$\ell_q$} quasinorm ratios have served as sparsity-inducing penalties in the long history of blind signal deconvolution or image deblurring, as stopping criteria (for instance in NMF, non-negative matrix factorization \cite{Hoyer_P_2004_j-mach-learn-res_non-negative_mfsc}),  measures of sparsity satisfying a number of sound parsimony-prone axioms \cite{Hurley_N_2009_j-ieee-tit_comparing_ms}, estimates for time-frequency spread or concentration \cite{Ricaud_B_2014_j-adv-comput-math_survey_upsspa}. 
The $\ell_p$ quasinorm weighted by $\ell_2$ is considered as a  ``possible sparseness criterion'' in \cite{Bronstein_A_2005_j-int-j-imag-syst-tech_sparse_icabstri} when scaled with the  normalizing factor $N^{\frac{1}{p}-\frac{1}{2}}$.  It bears close connection with the kurtosis \cite{Fisher_R_1930_j-proc-lond-math-soc_moments_pmsd} for centered distributions and central moments, widely used in sparse sensory coding \cite{Field_D_1994_j-neural-comput_what_gsc} or adaptive filtering \cite{Tanrikulu_O_1994_j-elec-letters_least-mean_knhosbafa}.

The most frequent one with  $(p,q) = (1,2)$ is used \cite{Krishnan_D_2011_p-cvpr_blind_dnsm,Repetti_A_2015_j-ieee-spl_euclid_tsbdsl1l2r} as a surrogate to $\ell_0$ \cite{NoseFilho_K_2014_p-eusipco_sparse_bdbsisl0n,Li_Y_2016_j-ieee-tit_identifiability_bdssc}. \cred{The work \cite{Repetti_A_2015_j-ieee-spl_euclid_tsbdsl1l2r} (SOOT: Smoothed \cred{$\ell_1$-over-$\ell_2$ norm} ratio) proposed an efficient optimization algorithm to address it with theoretically guaranteed convergence, with application to seismic signal retrieval.} This ratio was also used to enhance lasso recovery on graphs \cite{Bresson_X_2015_p-eusipco_enhanced_lrg}. Its early history includes the "minimum entropy deconvolution'' proposed in \cite{Wiggins_R_1978_j-geoexploration_minimum_ed}, where the "varimax norm", akin to  kurtosis $\left(\ell_4 / \ell_2 \right)^4$, is maximized to yield visually simpler (spikier) signals. It was inspired by simplicity measures proposed in factor analysis \cite{Ferguson_G_1954_j-psychometrika_concept_pfa}, and meant to improve one of the earliest mentioned $\ell_0$ regularization \cite{Claerbout_J_1973_j-geophysics_robust_med}  in seismic. The  relationship with the concept of entropy was  explained later  \cite{Donoho_D_1981_p-atsas_minimum_ed}. It was generalized to the so-called "variable norm deconvolution" by maximizing $\left(\ell_q / \ell_2 \right)^q$ \cite{Gray_W_1978_tr_variable_nd}.  Note that  techniques in \cite{Wiggins_R_1978_j-geoexploration_minimum_ed,Gray_W_1978_tr_variable_nd} are relatively rudimentary. They aim at finding some inverse filter that maximizes a given contrast. They do not explicitly take into account noise statistics. Even more, the  deconvolved  estimate is  linearly obtained from observations, see \cite{Castella_M_2010_incoll_convolutive_m} for an overview.
Recently,  \cite{Demanet_L_2014_j-inf-inference_scaling_lrses} uses \cred{$\ell_1$-over-$\ell_{\infty}$} for sparse  recovery, and \cite{Chang_X_2018_j-stat-sin_sparse_kmlil0phddc} \cred{$\ell_{\infty}$-over-$\ell_0$} for cardinality-penalized clustering. The family of entropy-based sparsity measures $\left(\ell_q/\ell_1 \right)^{\frac{q}{1-q}}$ \cite{Lopes_M_2016_j-ieee-tit_unknown_scsdi} (termed $q$-ratio sparsity level in \cite{Zhou_Z_2019_j-sp_sparse_rbqrcmsv}), extends a previous work on squared \cred{$\ell_1$-over-$\ell_2$} ratios \cite{Lopes_M_2013_p-icml_estimating_uscs} for compressed sensing. Finally, \cite{Yu_Y_2016_j-neurocomputing_concave-convex_nrbdmfabd} proposes an extension of \cite{Krishnan_D_2011_p-cvpr_blind_dnsm} to an  \cred{$\ell_q$-over-$\ell_2$} ratio to discriminate between sharp and blurry images, and \cite{Jia_X_2018_j-mech-syst-sp_sparse_fglplqnacmrm,Jia_X_2017_j-sp_geometrical_iglplqnbd} use a norm ratio for the purpose of impulsive signature enhancement in sparse filtering, still without rigorous convergence proofs. 


\subsection{Contribution and outline}
Our main contribution resides in providing a \cred{new family} of smooth-enough surrogates to \lz with sufficient Lipschitz regularity. The resulting penalties, called SPOQ, \cred{Smoothed $\ell_p$-over-$\ell_q$}, extend the \cred{SOOT, smoothed $\ell_1$-over-$ \ell_2$} \cred{ratio penalty from} \cite{Repetti_A_2015_j-ieee-spl_euclid_tsbdsl1l2r}. A novel trust-region algorithm \ldd{which} generalizes and improves the variable metric forward-backward algorithm from \cite{Chouzenoux_E_2014_j-optim-theory-appl_variable_mfbamsdfcf} \cred{is then introduced, to minimize the sum of the SPOQ penalty and a non-necessarily smooth convex term}. \cred{The convergence of our algorithm to a critical point of the problem is shown under mild technical assumptions.} \cred{We then show the applicability of our method on the practical problem of sparse signal retrieval arising in mass spectrometry data analysis. In particular, we illustrate the great ability of SPOQ penalty to retrieve accurately the sought spike position and intensities, when compared to a bunch of state-of-the-art sparse recovery methods.}  
Section~\ref{Introduction} recalled the parsimony role and introduces sparsity measures. Section~\ref{ProposedFormulation} describes the observation model and the proposed SPOQ quasinorm-norm 
ratio regularization. We then derive our trust-region minimization algorithm and analyze its convergence in Section~\ref{MinimizationAlgorithm}. Section~\ref{Application} illustrates the good performance of SPOQ regularization \cred{in recovering} ``naturally sparse" mass spectrometry signals, \cred{as compared to several} sparsity penalties \cred{for restoration and recovery}.

\section{Proposed formulation\label{ProposedFormulation}}

\subsection{Sparse signal reconstruction}
Let us consider the observation model
\begin{equation}
\label{eq:pbinverse}
\ey = \eD \ex + \eb
\end{equation}
where $\ey = (y_m)_{1 \leq m \leq M} \in \eR^M$ represents the degraded measurements related to the original signal $\ex=(x_n)_{1 \leq n \leq N} \in \eR^N$ through the observation matrix $\eD \in \eR^{M \times N}$.
Hereabove, $ \eb\in \eR^M$ models additive acquisition noise. In this work, we focus on the inverse problem aiming at recovering signal $\ex$ from $\ey$ and $\eD$, under the assumption that the sought signal is sparse, i.e., has few non-zero entries. 
A direct (pseudo) inversion of $\eD$ generally yields poor-quality solutions,  because of noise and  the ill-conditioning of $\eD$. \cred{More suitable is} a penalized approach, which defines an estimate $\hat{\ex} \in \eR^N$ of $\ex$ as a solution of the constrained minimization problem
\begin{equation}
\label{eq:pbminimzation1}
\minimize{\ex \in \eC} {\cred{\Psi(\ex)} + \Theta(\ex)},
\end{equation}
where $\eC$ is a non-empty convex and compact subset of $\eR^N$. Moreover, function $\Theta: \eR^N \to ]-\infty, +\infty]$ is a \cred{non-necessarily smooth convex} data fidelity function measuring the discrepancy between the observation and the model. One can define $\Theta$ as the least-squares \cred{function $\zeta : \ex \mapsto \xi\| \eD \ex -\ey \|^2$, } or adopt an interesting constrained formulation, by setting
\begin{equation}
(\forall \ex \in \eR^N) \quad
\Theta(\ex) =  \iota_{\mathcal{B}^\ey_\xi}(\eD\ex).
\label{eq:phidef}
\end{equation}
Hereabove, $\xi >0$ is a parameter depending on noise characteristics, $\mathcal{B}^\ey_\xi$ is the Euclidean ball centered at $\ey$ with radius $\xi$, and $\iota_{\mathcal{S}}$ denotes the indicator function of a set $\mathcal{S}$, equal to zero for $\ex \in \mathcal{S}$, and $+ \infty$ otherwise. Furthermore, function $\Psi$ \cred{defined from} $\eR^N \to ]-\infty, +\infty]$ is a regularization function used to enforce desirable properties on the solution. The choice of \ldd{the regularization function} $\Psi$ is essential for reaching satisfying results \cred{by promoting desirable properties in the sought signal.} When sparsity is expected, the $\ell_1$ norm is probably the most used regularization function. It is  a convex envelope proxy to \lz, a key feature for deriving  efficient minimizations to solve \eqref{eq:pbminimzation1}. However, because it is not scale invariant, the $\ell_1$ penalty can lead to an under-estimation bias of signal amplitudes, which is detrimental to the quality of the solution. In this work, we propose a new regularization strategy, relying on the \cred{$\ell_p$-over-$\ell_q$} norm ratio, aiming at limiting  scale ambiguity in the estimation. 

\subsection{Proposed SPOQ penalty}
Let $p \in \, ]0, 2[$ and $q \in [2, +\infty[$. We first define two smoothed approximations to $\ell_p$ and $\ell_q$ parametrized by constants $(\alpha, \eta)\in ]0,+\infty[^2$:
for every $\ex= (x_{n})_{1\le n \le N}\in \mathbb{R}^{N}$,
\begin{equation}
\allx
\ell_{p,\alpha}(\ex) = \left(\sum_{n=1}^N \left(\left(x_n^2 + \alpha^2\right)^{p/2}  - \alpha^p \right)\right)^{1/p}
\label{eq:lpsmooth}
\end{equation}
and
\begin{equation}
\allx
\ell_{q,\eta}(\ex) = \left( \eta^q + \sum_{n=1}^N |x_n|^q \right)^{1/q}.
\end{equation}
Remark that the traditional $p$ and $q$ (quasi-)norms are recovered for $\alpha = \eta = 0$. Our Smoothed  \cred{$\ell_p$-over-$\ell_q$} (SPOQ) penalty is then defined as the following function:
\begin{equation}
\allx
\Psi(\ex) = \log \left( \dfrac{(\ell_{p,\alpha}^p(\ex) + \beta^p)^{1/p}}{\ell_{q,\eta}(\ex) } \right).
\label{eq:reg_sparse} 
\end{equation}
Parameter $ \beta \in ]0,+\infty[$ is introduced to account for the fact that the $\log$ function is not defined at $0$. It is worth  noting that the proposed function \eqref{eq:reg_sparse}  combines both the sparsity promotion effect of the non-convex logarithmic loss \cite{Mazumder_R_2011_j-asa_sparsenet_cdnp,Selesnick_I_2014_j-ieee-tsp_sparse_semsco} and of the \cred{$\ell_p$-over-$\ell_q$} ratio. It generalizes the Euclidean/Taxicab Smoothed \cred{$\ell_1$-over-$\ell_2$} (SOOT) penalty \cite{Repetti_A_2015_j-ieee-spl_euclid_tsbdsl1l2r}, recovered for $p=1$ and $q=2$. Figure \ref{Fig:sparsity-promoting} illustrates the shape of the SPOQ penalty in the case $N=2$, for $p=1$ and $q=2$ (i.e., SOOT), and $p=1/4$ and $q=2$, in comparison with $\ell_0$, $\ell_1$, \cred{SCAD, and \celo}. It is worth noticing that, on the second row, the logarithm sharpens the \cred{$\ell_1$-over-$\ell_2$} behavior toward  \lz. By choosing $p=1/4$,  \spoq further enhances the folds along the axes. As a result, the bottom-right picture best mimics the top-left \lz representation.

\begin{figure}[htb]
\centering
\includegraphics[width=8cm]{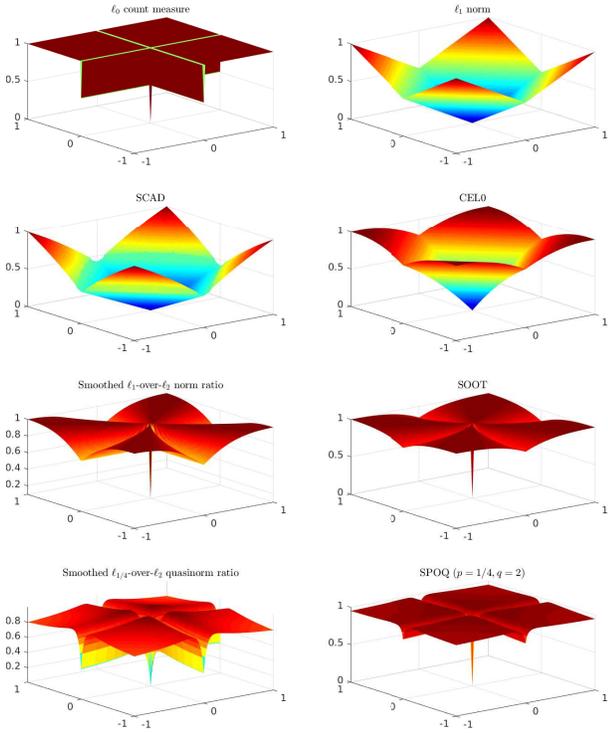}
\caption{Sparsity-promoting penalties, scaled to $[0,\,1]$. From top to bottom and from left to right: $\ell_0$, $\ell_1$, \cred{SCAD (with $\delta=2.5, a=1.1$), \celo (with $\delta=1$)}, smoothed \cred{$\ell_1$-over-$\ell_2$}, SOOT, smoothed \cred{$\ell_p$-over-$\ell_q$} and SPOQ with $(\alpha,\beta,\eta)=(\num{ 7e-7},\num{ 3e-3},\num{ 1e-1})$.}
\label{Fig:sparsity-promoting}
\end{figure}

\subsection{Mathematical properties}
We present here several properties of the proposed SPOQ penalty, that will be essential for deriving an efficient optimization algorithm to solve Problem~\eqref{eq:pbminimzation1}.

\subsubsection{Gradient and Hessian}
Let us express the gradient and Hessian matrices of function $\Psi$ at $\ex \in \eR^N$:
\begin{equation}
\begin{cases}
\nabla \ell_{q,\eta}^q(\ex) = q\, \left(\sign(x_n)|x_n|^{q-1}\right)_{1 \leq n \leq N} \\  
\nabla^2 \ell_{q,\eta}^q(\ex) = q(q-1)\,\DD{\left((|x_n|^{q-2})_{1 \leq n \leq N}\right)}
\end{cases}
\label{e:D1ellqx}
\end{equation}
and
\begin{equation}
\begin{cases}
\nabla \ell_{p,\alpha}^p(\ex) = p \, \left( x_n (x_n^2 + \alpha^2)^{\frac{p}{2}-1}\right)_{1\leq n \leq N} \\
\nabla^2 \ell_{p,\alpha}^p(\ex) =  p \,\DD \left((\left((p-1)x_n^2+\alpha^2\right) \right.\\
\left. \qquad \qquad \qquad\qquad\;\;\times (x_n^2+\alpha^2)^{\frac{p}{2}-2})_{1\le n \le N}\right),
\label{e:D2ellpx}
\end{cases}
\end{equation}
with the notation $\sign(x) = 0$ for $x = 0$, $-1$ for $x<0$ and $+1$ for $x>0$. It is worth noting that one can decompose the SPOQ penalty under the form 
\begin{equation}
\allx
 \Psi(\ex) = \Psi_1(\ex) - \Psi_2(\ex),
\end{equation}
by setting
\begin{equation}
\allx
\Psi_1(\ex) = \frac{1}{p} \log \left( \ell_{p,\alpha}^p(\ex) + \beta^p \right), 
\end{equation}
and
\begin{equation}
\allx
\Psi_2(\ex) = \frac{1}{q}  \log \left( \ell_{q,\eta}^{q}(\ex) \right).
\end{equation} 
Hence, $\nabla \Psi = \nabla \Psi_1 - \nabla \Psi_2$, and $\nabla^2 \Psi = \nabla^2 \Psi_1 - \nabla^2 \Psi_2$, with
\begin{align}
\nabla \Psi_1(\ex) & = \frac{1}{p} \frac{\nabla \ell_{p,\alpha}^p(\ex)}{\ell_{p,\alpha}^p(\ex)+\beta^p} ,
\label{e:Dphi1}
\\
\nabla \Psi_2(\ex) & = \frac{1}{q} \frac{\nabla \ell_{q,\eta}^q(\ex)}{\ell_{q,\eta}^q(\ex)}, 
\label{e:Dphi2}
\\
p\nabla^2 \Psi_1(\ex) &=  \frac{\nabla^2 \ell_{p,\alpha}^p(\ex)}{ \ell_{p,\alpha}^p(\ex)+\beta^p} -  \frac{ \nabla \ell_{p,\alpha}^p(\ex)\left(\nabla \ell_{p,\alpha}^p(\ex)\right)^\top}{\left( \ell_{p,\alpha}^p(\ex)+\beta^p \right)^2},
\label{e:D2phi1}
\\
q\nabla^2 \Psi_2(\ex)  &=  \frac{\nabla^2 \ell_{q,\eta}^q(\ex)}{\ell_{q,\eta}^q(\ex)} - \frac{ \nabla \ell_{q,\eta}^q(\ex) \left(\nabla \ell_{q,\eta}^q(\ex) \right)^\top}{\ell_{q,\eta}^{2q}(\ex)}  .
\label{eq_phi-hessian1}
\end{align}

From the above\ldd{ expressions}, we derive the following proposition, stating that for suitable parameter choices, function $\Psi$ has $\zerob_N$, i.e. the zero vector of dimension $N$, as a minimizer, which is desirable for a sparsity promoting regularization function. 

\begin{proposition}\label{prop:minimiseur}
Assume that either $q = 2$ and $\eta^2 \alpha^{p-2} > \beta^p$, or $q > 2$. Then, $\nabla^2 \Psi(\zerob_N)$ is a positive definite matrix and $\zerob_N$ is a local minimizer of $\Psi$. In addition, if
\begin{equation}\label{e:simpcondglob}
\eta^2 \ge \beta^2\max\left\{\frac{8\alpha^{2-p}}{p(2+p)\beta^{2-p}},\frac{1}{(2^{p/2}-1)^{2/p}}\right\}
\end{equation}
then $\zerob_N$ is a global minimizer of $\Psi$.
\end{proposition}
\begin{proof} See Appendix \ref{Proof0}.
 \end{proof}


\subsubsection{Majorization properties}
We now gather in the following proposition two properties that allow us to build quadratic surrogates for Function \eqref{eq:reg_sparse}. 
\begin{proposition} Let $\Psi$ be defined by \eqref{eq:reg_sparse}.\\
\label{proposPsi}
(i) $\Psi$ is a $L$-Lipschitz differentiable function on $\eR^N$, \textit{i.e}, for every $(\ex,\ex')\in (\mathbb{R}^{N})^{2}$,
\begin{equation}
\allxx
 \| \nabla \Psi(\ex) - \nabla \Psi(\ex') \| \leq L \|\ex-\ex'\|
\end{equation}
where 
\begin{equation}
\label{eq:LipSPOQ}
L= p\frac{\alpha^{p-2}}{\beta^p} +  \frac{p}{2\alpha^2} \max\Big\{1,\Big(\frac{N\alpha^p}{\beta^p}\Big)^{2}\Big\} + \frac{q-1}{\eta^2}.
\end{equation} 
In particular, 
\begin{multline}
\allxx
\Psi(\ex') \leq \Psi(\ex) + (\ex'-\ex)^\top \nabla \Psi(\ex) + \frac{L}{2} \|\ex'-\ex\|^2.
\end{multline}
(ii) For every $\rho \in [0, +\infty[$, define the $\ell_q$-ball complement:
\begin{equation}
\overline{\BB}_{q,\rho} = \{\ex = (x_{n})_{1\le n \le N} \in \eR^N \mid \sum_{n=1}^N |x_n|^q \ge \rho^q\}.\label{eq_outer-lq-ball}
\end{equation}
$\Psi$ admits a quadratic tangent majorant at every $\ex \in \overline{\BB}_{q,\rho}$, i.e.
\begin{multline}\label{e:majloc}
(\forall \, \ex'\in \overline{\BB}_{q,\rho} ) \quad  \Psi(\ex') \leq \Psi(\ex) +
(\ex'-\ex)^\top \nabla \Psi(\ex) \\ + \frac12 (\ex'-\ex)^\top \eA_{q, \rho}(\ex) (\ex'-\ex), 
\end{multline}
where
\begin{multline}
\eA_{q, \rho}(\ex) = \chi_{q,\rho}\, \boldsymbol{I}_N \\
+ \frac{1}{\ell_{p,\alpha}^p(\ex)+\beta^p} \DD{ \big((x_n^2+\alpha^2)^{p/2-1}\big)_{1\le n \le N}},
\label{eq:metriqueA}
\end{multline}
with
\begin{equation}
\chi_{q,\rho}=\frac{q-1}{(\eta^q+\rho^q)^{2/q}}.
\end{equation}
Moreover, for every $\ex \in \eR^N$,
\begin{equation}
\chi_{q,\rho} \, \eI_N \cred{\preceq} \eA_{q, \rho}(\ex) \cred{\preceq} (\chi_{q,\rho} + \beta^{-p} \alpha^{p-2}) \, \eI_N. \label{eq:maj_bnd}
\end{equation}
\end{proposition}

\begin{proof}
See Appendix \ref{Proof1}.
\end{proof}

Proposition \ref{proposPsi}(i) leads to a rather simple majorizing function for $\Psi$, valid on the whole Euclidean space $\eR^N$. This extends our previous result established in \cite{Repetti_A_2015_j-ieee-spl_euclid_tsbdsl1l2r} for the particular case when $p=1$ and $q=2$. The majorization property presented in Proposition \ref{proposPsi}(ii) only holds in the non-convex set $\overline{\BB}_{q,\rho}$.
By limiting the size of the region where majorization is imposed, one may expect more accurate approximations for $\Psi$. This observation motivates the trust-region minimization algorithm \cred{that} we propose in the next section to solve Problem \eqref{eq:pbminimzation1}.

\section{Minimization algorithm}
\label{MinimizationAlgorithm}

\subsection{Preliminaries}

We first introduce some key notation and concepts. 
Problem \eqref{eq:pbminimzation1} can be rewritten equivalently as:
\begin{equation}
\label{eq:pbminimzation2}
\minimize{\ex \in \eR^N} {\Omega(\ex)}
\end{equation}
where $\Omega = \Psi + \Phi$ with $\Psi$ defined in \eqref{eq:reg_sparse} and $\Phi = \Theta + \iota_{\eC}$. We will assume that function $\Theta$ belongs to $\Gamma_0(\eR^N)$, the class of convex lower semi-continuous functions. This is for instance valid for the least-squares term as well as for function \eqref{eq:phidef}. Note that the assumptions made on set $\eC$ implies that $\Phi$ is coercive and belongs to $\Gamma_0(\eR^N)$.
The particular structure of $\Omega$, summing a Lipschitz differentiable function $\Psi$ and the non-necessarily smooth convex term $\Phi$ \cred{suits it well to} the class of variable metric forward-backward (VMFB) optimization methods \cite{Becker_S_2012_p-nips_quasi-newton_psm, Combettes_P_2014_j-optimization_variable_mfbsamid,Chouzenoux_E_2014_j-optim-theory-appl_variable_mfbamsdfcf,Repetti_A_2019_PREPRINT_variable_mfbacmp}.
In such methods, one alternates gradient steps on $\Psi$ and proximity steps on $\Phi$, preconditioned by a specific sequence of metric matrices. Let us recall that, for $\Phi \in \Gamma_0(\eR^N)$, and for a symmetric positive definite (SPD) \cred{metric matrix} $\eA \in \eR^{N \times N}$, the proximity operator of $\Phi$ at $\ex \in\eR^N$ relative to the metric \cred{matrix} $\eA$ is defined as
\begin{equation}
\label{eq:proxmetrique}
\prox_{\eA,\Phi}(\ex) = \argmin{\ez \in \eR^N} {\left(\frac{1}{2}\|\ez-\ex\|_{\eA}^2 + \Phi(\ez)\right)}
\end{equation}                                                                                            
with  notation $\| \mathbf{u} \|_{\eA} = \sqrt{\mathbf{u}^\top \eA \mathbf{u}}$ for $\mathbf{u} \in \eR^N$. 
Then, the VMFB method for solving Problem \eqref{eq:pbminimzation2} reads, for every $k \in \eN$,
\begin{equation}
\ex_{k+1 } = \prox_{\gamma_k^{-1} \eA_{k}, \Phi}\left(\ex_k - \gamma_k (\eA_{k})^{-1} \nabla \Psi(\ex_k) \right),
\label{eq:proxineexact}
\end{equation}
where $\ex_0 \in \eR^N$, and $(\gamma_k)_{k \in \eN}$ and $(\eA_{k})_{k \in \eN}$ are sequences of positive stepsizes and SPD metric \cred{matrix}, respectively, chosen in such a way to guarantee the convergence of VMFB iterates to a solution to Problem \eqref{eq:pbminimzation1} \cite{Chouzenoux_E_2014_j-optim-theory-appl_variable_mfbamsdfcf}. Two main challenges arise, when implementing the VMFB algorithm, namely (i) the choice for the preconditioning matrices $(\eA_{k})_{k \in \eN}$, and (ii) the evaluation of the proximity operator involved in the update \eqref{eq:proxineexact}. In \cite{Chouzenoux_E_2014_j-optim-theory-appl_variable_mfbamsdfcf}, a novel methodology was proposed based on the choice of preconditioning matrices satisfying a majorization condition for $\Psi$. This methodology provides a practically efficient algorithm. Furthermore, it allows to establish convergence in the case of a non necessarily convex function $\Psi$, as soon as it satisfies the so-called Kurdyka-\L{}ojasewiecz inequality \cite{Attouch_H_2013_j-math-programm_convergence_dmsatppafbsrgsm}. Convergence also holds when the proximity update is subject to numerical errors. These advantages are particularly beneficial in our context, as our SPOQ penalty $\Psi$ is non-convex, and the data fidelity term $\Phi$ may have a non closed form for its proximity operator (for instance, in the case of \eqref{eq:phidef}). As shown in Proposition \ref{proposPsi}, function $\Psi$ is Lipschitz differentiable and thus a constant metric could be used in our implementation of VMFB, then reduced to the standard forward-backward (FB) scheme. However, FB algorithm \cred{sometimes exhibit poor convergence speed performance}. In particular, it is clear that $L$ --- the Lipschitz constant defined in \eqref{eq:LipSPOQ}, albeit an upper bound --- can become very high for small parameters $(\alpha,\beta,\eta)$, which is actually the case of interest as they only act as smoothing constants for ensuring differentiability\ldd{ of the function}. As shown in Proposition \ref{proposPsi}(ii), it is possible to build a more accurate quadratic majorizing approximation on $\Psi$, whose curvature depends on the point it is calculated. However, the majorization in that case holds only on a subset of $\eR^N$. \cred{We extend  \cite{Chouzenoux_E_2014_j-optim-theory-appl_variable_mfbamsdfcf} with a trust-region scheme in order to use local majorizing metrics.} 
Without deteriorating the convergence guarantees of the original method, this gives rise to a novel preconditioned proximal gradient scheme adapted at each iteration.  



\subsection{Proposed algorithm}
\cred{In this section, we present our Algorithm \ref{alg:TRVMFB} called trust-region (TR) VMFB, for solving Problem \eqref{eq:pbminimzation2}}. At each iteration $k \in \eN$, we will make $B \geq 1$ trials of values $(\rho_{k,i})_{k \in \eN, 1 \leq i \leq B}$ for the trust-region radius. For each tested radius $\rho_{k,i} \geq 0$, a VMFB update $\ez_{k,i}$ is computed within the majorizing metric \cred{matrix} $\eA_{q, \rho_{k,i}}$ at $\ex_k$, defined in Proposition \ref{proposPsi}(ii). Then, a test is performed for checking whether the update does belong to the region $\overline{\BB}_{q,\rho_{k,i}}$. If not, the region size is reduced with a factor $\theta \in ]0,1[$, and a new VMFB step is performed. The trust-region loop stops as soon as $\ez_{k,i} \in \overline{\BB}_{q,\rho_{k,i}}$. Note that, for the last trial, i.e. $i = B$, a radius equal to $0$ is tested, which allows us to guarantee \cred{that our method is well-defined}. More precisely, this leads to the following sequence, for the radius values: 
\begin{equation}                      
\label{eq:rhoTR}
\rho_{k,i} =
\begin{cases}
 \sum_{n=1}^N |x_{n,k}|^q & \text{if} \quad i = 1\\
\theta \rho_{k,i-1} & \text{if} \quad 2 \leq i \leq B-1 \\
0 & \text{if} \quad i = B.
\end{cases}
\end{equation}
Let us remark that $\ex_k \in \overline{\BB}_{q,\rho_{k,1}}$ and the following inclusion holds by construction:
\begin{equation}
\label{eq:BBinc}
 \overline{\BB}_{q,\rho_{k,1}} \subset \overline{\BB}_{q,\rho_{k,2}} \cdots \subset \overline{\BB}_{q,\rho_{k,B}} = \eR^N.
\end{equation}
\cred{Thus, for every $i \in \{1,\ldots,B\}$, $\ex_k \in  \overline{\BB}_{q,\rho_{k,i}}$ and the Proposition \ref{proposPsi}(ii), ensuring the validity of the local majorizing quadratic, can be applied.}

\begin{algorithm}
\caption{TR-VMFB algorithm}
\label{alg:TRVMFB}
$$
\begin{array}{l}
\text{Initialize:} \quad \ex_0 \in \text{dom} \Phi, \; B \in \eN^*, \; \theta \in ]0,1[, \; (\gamma_k)_{k \in \eN} \in ]0,+\infty[\\
\text{For}\;\;k = 0,1,\ldots:\\
\left\lfloor
\begin{array}{l}
\text{For}\;\;i = 1,\ldots,B:\\
\left\lfloor
\begin{array}{l}
\text{Set trust-region radius} \quad \rho_{k,i} \quad \text{using \eqref{eq:rhoTR}}\\
\text{Construct} \quad \eA_{k,i} = \eA_{q,\rho_{k,i} }(\ex_k) \quad \text{using \eqref{eq:metriqueA}}\\
\ez_{k,i} = \prox_{\gamma_k^{-1} \eA_{k,i}, \Phi}\left(\ex_k - \gamma_k (\eA_{k,i})^{-1} \nabla \Psi(\ex_k) \right)\\
\text{If}\;\; \ez_{k,i} \in \overline{\BB}_{q,\rho_{k,i}}: \text{Stop loop}\\
\end{array} 
\right. \\ 
\ex_{k+1} = \ez_{k,i}\\
\end{array}
\right.
\end{array} 
$$
\end{algorithm}

As already mentioned, the computation of the proximity operator of $\Phi$ within a general SPD metric \cred{matrix} cannot usually be performed in a closed form, and an inner solver is required. In order to encompass this situation, we propose in Algorithm~\ref{alg:TRVMFBi} an inexact form of our TR-VMFB method. The precision for the computation of the proximity update is measured by means of two inequalities, Alg.~\ref{alg:TRVMFBi}(a) and Alg.~\ref{alg:TRVMFBi}(b). 
\cred{In practice, an inner solver such as \cite{Pesquet_J_2012_j-pacific-j-opt_parallel_ipom,Abboud_F_2017_j-math-imaging-vis_dual_bcfbaaddvs} can be employed to compute this update. Moreover, the outer loop of TR-VMFB would be ran until a given stopping criterion is met, as we will explain in Section \ref{sec:appli_B}.}

\begin{algorithm}
\caption{TR-VMFB algorithm --- Inexact form}
\label{alg:TRVMFBi}
$$
\begin{array}{l}
\text{Initialize:} \quad \ex_0 \in \text{dom} \Phi, \; B \in \eN^*, \; \theta \in ]0,1[, \; (\gamma_k)_{k \in \eN} \in ]0,+\infty[\\
\text{For}\;\;k = 0,1,\ldots:\\
\left\lfloor
\begin{array}{l}
\text{For}\;\;i = 1,\ldots,B:\\
\left\lfloor
\begin{array}{l}
\text{Set trust-region radius} \quad \rho_{k,i} \quad \text{using \eqref{eq:rhoTR}}\\
\text{Construct} \quad \eA_{k,i} = \eA_{q,\rho_{k,i} }(\ex_k) \quad \text{using \eqref{eq:metriqueA}}\\
\text{Find } \ez_{k,i} \in \eR^N  \text{ such that}\\
\text{(a)}\; \Phi(\ez_{k,i}) + (\ez_{k,i}-\ex_k)^\top \nabla \Psi(\ex_k) \\ 
\quad \quad \quad +  \gamma_k^{-1} \|\ez_{k,i} - \ex_k \|_{\eA_{k,i}}^{2} \leq \Phi(\ex_k)\\
\text{(b)}\; \| \nabla \Psi(\ex_k) + \er_{k,i} \| \leq \kappa \| \ez_{k,i} - \ex_k \|_{\eA_{k,i}} \\ 
\qquad \text{with} \; \er_{k,i} \in \partial \Phi(\ez_{k,i}) \;\text{and}\; \kappa>0 \\
\text{If}\;\; \ez_{k,i} \in \overline{\BB}_{q,\rho_{k,i}}: \text{Stop loop}\\
\end{array} 
\right. \\ 
\ex_{k+1} = \ez_{k,i}\\
\end{array}
\right.
\end{array} 
$$
\end{algorithm}

\subsection{Convergence analysis}
In this section, we show that Algorithm \ref{alg:TRVMFB} can be viewed as a special instance of Algorithm \ref{alg:TRVMFBi} provided that $\kappa$ is chosen large enough. Moreover, we establish a descent lemma for Algorithm \ref{alg:TRVMFBi}, that allows us to deduce its convergence to a solution to Problem \eqref{eq:pbminimzation2}. 
We start with the following assumptions on the sequences $(\gamma_k)_{k \in \eN}$ and $(\eA_{k,i})_{k \in \eN,1\leq i \leq B}$, that are necessary for our convergence analysis:

\begin{assumption} ~ \\
(i) There exist\lda{s} $(\underline{\gamma},\overline{\gamma}) \in ]0, +\infty[^2$ such that for every $k \in \eN$, \\
$\underline{\gamma} \leq \gamma_k \leq 2 - \overline{\gamma}$.\\
(ii) There exists $(\underline{\nu}, \overline{\nu}) \in ]0, +\infty[^2$, such that, for every $k \in \eN$ and for every $i \in \{1,\ldots,B\}$, $
\underline{\nu} \eI_N \cred{\preceq} \eA_{k,i} \cred{\preceq} \overline{\nu}\eI_N$. 
\label{ass:Hyp}
\end{assumption}
\begin{remark}
By construction,  iterates $(\ex_k)_{k \in \eN}$ produced by Algorithms \ref{alg:TRVMFB} and \ref{alg:TRVMFBi}, belong to the domain of $\Phi$ \cred{and therefore} to the set $\eC$. This implies that sequence $(\ex_k)_{k \in \eN}$ is bounded, so that there exists $\rho_{\max} \geq 0$ such that, for every $k \in \eN$ and  $i \in \{1,\ldots,B\}$, we have $\rho_{k,i} \leq \rho_{\max}$. Assumption~\ref{ass:Hyp}(ii) thus holds as a consequence of \eqref{eq:maj_bnd}, by setting $\underline{\nu} = \chi_{q,\rho_{\max}}$ and $\overline{\nu} =  \chi_{q,0} + \beta^{-p} \alpha^{p-2}$. 
\end{remark}

The following lemma establishes the link between Algorithm \ref{alg:TRVMFB} and its inexact form, Algorithm \ref{alg:TRVMFBi}.
\begin{lemma}~ \\
Under Assumption \ref{ass:Hyp}, for every $i \in \{1,\ldots,B \}$, there exist $\er_{k,i} \in \partial \Psi(\ez_{k,i})$
such that conditions Alg.~\ref{alg:TRVMFBi}(a) and Alg.~\ref{alg:TRVMFBi}(b) are fulfilled, with 
\begin{equation}
\label{eq:exactprox}
\ez_{k,i} = \prox_{\gamma_k^{-1} \eA_{k,i}, \Phi}\left(\ex_k - \gamma_k (\eA_{k,i})^{-1} \nabla \Psi(\ex_k) \right)
\end{equation}
and $\kappa \geq \underline{\gamma}^{-1} \sqrt{\overline{\nu}}$.
\label{prop:2}
\end{lemma}
\begin{proof} Let $k \in \eN$ and $i \in \left\{1,\ldots,B\right\}$, and set $\ez_{k,i}$ as in \eqref{eq:exactprox}. 
Due to the variational definition of the proximity operator, and the convexity of $\Phi$,  there exist\afr{s}{\ldd{s}} $\er_{k,i} \in \partial \Phi(\ez_{k,i})$ such that
\begin{equation}
\label{eq:proxopt}
\begin{cases}
\er_{k,i} = - \nabla \Psi(\ex_k) + \gamma_k^{-1} \eA_{k,i} (\ex_k - \ez_{k,i})\\
(\ez_{k,i} - \ex_k)^\top \er_{k,i} \geq \Phi(\ez_{k,i}) - \Phi(\ex_k).
\end{cases}
\end{equation}
Thus, $\ez_{k,i}$ satisfies:
\begin{equation}
\Phi(\ez_{k,i}) + (\ez_{k,i} - \ex_k)^\top \nabla \Psi(\ex_k) + \gamma_k^{-1} \| \ez_{k,i} - \ex_k \|^2_{\eA_{k,i}} \leq \Phi(\ex_k)
\end{equation}
Therefore, condition Alg.~\ref{alg:TRVMFBi}(a) holds. Moreover, using \eqref{eq:proxopt} and Assumption \ref{ass:Hyp}, 
\begin{align}
\| \er_{k,i} + \nabla \Psi(\ex_k) \| & = \gamma_k^{-1} \| \eA_{k,i} (\ex_k - \ez_{k,i}) \| \nonumber\\
& \leq \underline{\gamma}^{-1} \sqrt{\overline{\nu}} \| \ex_k - \ez_{k,i} \|_{\eA_{k,i}}.
\end{align}
Hence the condition Alg.~\ref{alg:TRVMFBi}(b) holds for $\kappa \geq \underline{\gamma}^{-1} \sqrt{\overline{\nu}}$.
\end{proof}
We now establish a descent property on the sequence generated by our method. 


\begin{lemma} ~ \\
\label{lem:desc}
Under Assumption \ref{ass:Hyp}, there exists $\mu \in ]0,+\infty[$ such that, for every $k \in \eN$,
\begin{equation}
\label{eq:des1}
\Omega(\ex_{k+1}) \leq \Omega(\ex_k) - \frac{\mu}{2}\| \ex_{k+1} - \ex_k \|^2 
\end{equation}
with $(\ex_k)_{k \in \eN}$ defined in Algorithm \ref{alg:TRVMFBi}. 
\end{lemma}

\begin{proof}
We have
\begin{equation}
(\forall k \in \eN) \quad \Omega (\ex_{k+1}) = \Psi(\ex_{k+1}) + \Phi(\ex_{k+1})
\end{equation}
Under condition Alg.~\ref{alg:TRVMFBi}(a),
\begin{equation}
\Phi(\ex_{k+1}) + (\ex_{k+1}-\ex_k)^\top \nabla \Psi(\ex_k) + \gamma_k^{-1} \|\ex_{k+1} - \ex_k \|_{\eA_{k,i}}^{2} \leq \Phi(\ex_k).
\end{equation}
By construction, $\ex_{k+1} \in \overline{\BB}_{q,\rho_{k,i}}$ for some $i\in \{1,\ldots,B\}$. Moreover, $\ex_k \in \overline{\BB}_{q,\rho_{k,1}} \subset \overline{\BB}_{q,\rho_{k,i}}$. Therefore, by Proposition~\ref{proposPsi},
\begin{equation}
\label{eq:Fxplus1}
\Psi(\ex_{k+1}) \leq \Psi(\ex_k)+(\ex_{k+1}-\ex_k)^\top \nabla \Psi(\ex_k) + \frac{1}{2}\|\ex_{k+1}-\ex_k\|_{\eA_{k,i}}^{2}.
\end{equation}
Thus,
\begin{align}
\Omega(\ex_{k+1}) 
& \leq \Psi(\ex_k) + \Phi(\ex_k) + \frac{1}{2}\|\ex_{k+1} -\ex_k \|_{\eA_{k,i}}^2  \nonumber \\
& \qquad  - \gamma_k^{-1} \|\ex_{k+1}- \ex_k\|_{\eA_{k,i}}^2,   \nonumber\\
& \leq \Omega(\ex_k) - (\gamma_k^{-1} - \frac{1}{2}) \|\ex_{k+1} -\ex_k \|_{\eA_{k,i}}^2 .
\end{align}
Consequently, using Assumption \ref{ass:Hyp}, we deduce \eqref{eq:des1} by taking  $\mu = \frac{\underline{\nu} \overline{\gamma}}{2(2-\overline{\gamma})}$.
\end{proof}


\begin{theorem}~ \\
If $\Phi$ is a semi-algebraic function on $\eR^N$ and Assumption \ref{ass:Hyp} holds, then the sequence $(\ex_k)_{k \in \eN}$ generated by Algorithm \ref{alg:TRVMFB} converges to a critical point $\hat{\ex}$ of $\Omega$.
\end{theorem}

\begin{proof}
As already stated, the compactness assumption made on $\eC$ implies that $\Omega$ is coercive. Moreover, it belongs to an o-minimal structure including semi-algebraic functions and logarithmic function, so that it satisfies Kurdyka-\L{}ojasiewicz inequality \cite{Attouch_H_2013_j-math-programm_convergence_dmsatppafbsrgsm,VanDenDries_L_1998_book_tame_toms}. Therefore, by using Lemma \ref{lem:desc}, and \cite[Theorem 4.1]{Chouzenoux_E_2014_j-optim-theory-appl_variable_mfbamsdfcf}, we deduce that $(\ex_k)_{k \in \eN}$ converges to a critical point of $\Omega$.
\end{proof}

\section{Application to mass spectrometry processing}
\label{Application}

\subsection{Problem statement}
In this section, we illustrate the usefulness of the proposed SPOQ regularizer in the context of mass spectrometry (MS) data processing.
MS is a fundamental technology of analytical chemistry to identify, quantify, and extract important information on molecules \cred{from pure samples and  complex chemical mixtures}.
Thanks to its high performance and capabilities, MS is applied as a routine experimental procedure in several fields, including clinical research 
\cite{Tan_K_2012_j-clin-microbiol_prospective_emalditfmsshcmlibybbbsaitice}, anti-doping and proteomics \cite{Aebersold_R_2003_j-nature_mass_sbp}, metabolomics 
\cite{Scalbert_A_2009_j-metabolomics_mass-spectrometry-based_mlrf}, biomedical and biological analyses \cite{Mano_N_2003_j-anal-sci_biomedical_bms,SchmittKopplin_P_2003_j-electrophoresis_capillary_ems15yda},
diagnosis process, cancer and tumors profiling \cite{Schwartz_S_2004_j-clin-cancer-res_protein_pbtums}, food contamination detection 
\cite{Panchaud_A_2012_j-proteomics_mass_snphafbhep}.

In an MS experiment, the raw signal arising from the molecule ionization in an ion beam is measured as a function of time via Fourier Transform. A spectral analysis step is then performed leading to the so-called MS spectrum signal. It presents a set of positive-valued peaks distributed according to the charge state and the isotopic distribution of the studied molecule, generating typical patterns. \cred{The observed signal entails the determination of the  most probable sample chemical composition}, through the determination of the monoisotopic mass, charge state, and abundance of each present isotope. 

In the particular context of proteomic analysis, the studied chemical compound contains only molecules involving carbon, hydrogen, oxygen, nitrogen, and sulfur. Thus, its isotopic pattern at a given mass and charge state can be easily synthesized, by making use of the so-called "averagine"\footnote{Determining an \emph{average amino acid} from a statistical distribution.} model \cite{Senko_M_1994_j-anal-chem_collisional_almciuftms,Senko_M_1995_j-am-soc-mass-spectrom_determination_mmiplbrid}. Assuming that the charge state is known and mono-valued (see \cite{Cherni_A_2018_p-icassp_fast_dbamsda} for the multi-charged case), we propose to express the measured MS spectrum $\ey \in \eR^M$ as the sparse combination of individual isotopic patterns, i.e.
\begin{equation}
\ey = \sum_{n=1}^N x_n \ed(m_n^{\text{iso}}, z) + \eb
\label{eq:MSpb}
\end{equation}
where $\ed(m_n^{\text{iso}}, z) \in [0,+\infty[^M$ represents the mass distribution built with the "averagine" model at isotopic mass $m_n^{\text{iso}}$ and charge $z$, discretized on a grid of size $M$, and $x_n \geq 0$ the associated weight. A non-zero value for entry $x_n$ corresponds to the presence of monoisotope with mass $m_n^{\text{iso}}$. Moreover, $\eb \in \eR^M$ models the acquisition noise and some possible errors arising from the spectral analysis step. Let us form a dictionary matrix $\eD \in \eR^{M \times N}$ whose $n$-th column reads $\ed(m_n^{\text{iso}},z)$. Then, the above observation model \eqref{eq:MSpb} reads as \eqref{eq:pbinverse}, and the problem becomes the restoration of the sparse positive-valued signal $\ex$, given $\ey$ and $\eD$. We proposed in \cite{Cherni_A_2018_p-icassp_fast_dbamsda} \cred{a restoration based on a penalized least squares problems with $\ell_1$ prior and a primal-dual splitting minimization}. In this section, we show by means of several experiments the benefits obtained by considering instead the proposed SPOQ penalty. We also perform comparisons between SPOQ and 
various other non-convex penalties.

\subsection{Simulated datasets and settings}
\label{sec:appli_B}
Two synthetic signals A and B, with size $N = 1000$, are used for the sought vector $\ex$, containing $P$ randomly selected nonzero components ($P=48$ and $P=94$, respectively). In both examples, the mass axis contains $N$ regularly spaced values between $m_{\min}=1000$ Daltons and $m_{\max}=1100$ Daltons, and we set $M=N$. This allows us to generate the associated dictionary $\eD$. The condition number of this matrix is equal to $4\times 10^{4}$. The observed vector $\ey$ is then deduced using Model \eqref{eq:pbinverse}, where the noise is assumed to be zero-mean Gaussian, i.i.d with known standard deviation $\sigma$ (chosen as a given percentage of the MS spectrum maximal amplitude).

Figure \ref{Fig:sigAandB} presents the sought isotopic distributions $\ex$ and an example of associated MS spectra, for dataset A and B.
In order to retrieve the original sparse signals, we will solve Problem \eqref{eq:pbminimzation2} using $\Theta$ defined in \eqref{eq:phidef} and $\eC = [0,x_{\max}]^N$ with $x_{\max} = 10^5$. Concerning the regularization function $\Psi$, we will make comparisons between the $\ell_1$ norm, $\ell_0$, \cred{the Smoothly Clipped Absolute Deviation (SCAD) penalty $\Psi(\ex) = \sum_{n=1}^N \psi(x_n)$ with, for every $x \in \mathbb{R}$, $\psi(x) =  \delta |x|$ if $|x| < \delta$, $(2 a \delta |x| - x^2-\delta^2)/2(a-1)$ if $\delta \leq |x| < a \delta$, and $(a+1) \delta^2 /2$ if $|x| \geq a \delta$ defined for $\delta >0$ and $a>2$ \cite{Fan_J_2001_j-asa_variable_snplop}}, the SPOQ penalty for $p \in \left\{\afa{0.05, 0.1, 0.15, 0.2}, 0.25,0.5,0.75,1,1.25,1.5\right\}$ and $q \in \left\{2,3,\ldd{4,}5,10\right\}$, the Cauchy penalty $\Psi(\ex) = \sum_{n=1}^N \log(1 + x_n^2/\delta^2)$ with $\delta>0$ \cite[p. 111--112]{Rey_W_1983_book_introduction_rqrsm}, the Welsch penalty $\Psi(\ex) = \sum_{n=1}^N (1 - \exp(-x_n^2/\delta^2))$ with $\delta>0$ \cite{Chouzenoux_E_2013_j-siam-j-imaging-sci_majorize-minimize_sal2l0ir}, and the Continuous Exact $\ell_0$ penalty (\celo) $\Psi(\ex) =$\linebreak $\sum_{n=1}^N \big(\delta - \frac{\|\ed_n\|^2}{2} \left( |x_n| - \frac{\sqrt{2 \delta}}{\|\ed_n\|} \right)^2 \mathbf1_{\{|x_n|<\frac{\sqrt{2 \delta}}{\|\ed_n\|} \} }\big)$ where $\delta >0$ and $\ed_n$ is the $n$-th column of $\eD$ \cite{Soubies_E_2015_j-siam-j-imaging-sci_continuous_el0pcel0lsrp,Soubies_E_2019_j-math-imaging-vis_insights_ocl2l0mp},\footnote{\cred{The characteristic function is} defined as $\mathbf1_{\chi} = 1$ if \ldd{condition} $\chi$ holds, $0$ otherwise.}

The resolution of \eqref{eq:pbminimzation2} is performed by using the primal-dual splitting algorithm in \cite{Chambolle_A_2011_j-math-imaging-vis_first_opdacpai,Komodakis_N_2015_j-ieee-spm_playing_dorpdaslsop} in the case of $\ell_1$ norm, $\ell_0$, \cred{SCAD} and $\celo$ penalties. For Cauchy and Welsch penalties, we use the VMFB strategy, using the majorizing metrics described in \cite{Chouzenoux_E_2014_j-optim-theory-appl_variable_mfbamsdfcf}. Finally, in the case of SPOQ, we run our trust-region VMFB method, where we set $\theta =0.5$, $B=10$, and $\gamma_k \equiv 1.9$. The proximity operator of $\Phi$ within the metric \cred{matrix} is computed by using the parallel proximal splitting algorithm from~\cite{Pesquet_J_2012_j-pacific-j-opt_parallel_ipom}, with a maximum number of $5 \cdot 10^3$ iterations. With the exception of $\ell_1$, all the tested penalization potentials are non-convex and only convergence to a local minimum can be guaranteed. In order to limit the sensitivity to spurious local minima, we initialize the optimization method using $10$ iterations of primal-dual splitting algorithm with $\ell_1$ penalty. All algorithms were run until the stopping criterion defined as $\|\ex_{k+1} - \ex_k \| \le \epsilon \| \ex_k \| $ \cred{with $\epsilon = 10^{-4}$} is satisfied, and a maximum of $10^3$ iterations. The most difficult task in this application is to estimate the support of the signal. Each of the iterative approaches presented has been evaluated with this regard. In order to avoid any bias in the estimation of the signal values, the support estimation process has been followed by a basic least squares step.

The considered non-convex regularizations depend on smoothing parameters, namely $\delta$ for Cauchy, Welsch and \celo, and $(\alpha,\beta,\eta)$ for SPOQ. When not precised, hyperparameters were optimized with grid search to maximize the signal-to-noise ratio (SNR) defined as
\begin{equation}
\text{SNR}(\ex,\hat{\ex})= 20 \log_{10} \left( \frac{\|\ex\|_2}{\|\ex - \hat{\ex}\|_2}\right)
\label{eq:SNRL2}
\end{equation}
where $\hat{\ex}$ is the estimated signal and $\ex$ the original one. Moreover, the bound $\xi$ in \eqref{eq:phidef} is set to $\sqrt{N}\sigma$. A sensitivity analysis is performed to assess the influence of these parameters on the solution quality. \cred{For quantitative comparisons}, we use the SNR defined above, the thresholded SNR metric denoted TSNR, defined as the SNR computed only on the support of the sought sparse signal, and the sparsity degree given as the number of entries of the restored signal greater (in absolute value) than a given threshold (here we take $10^{-4}$). Figure \ref{Fig:sigAandB} shows the difficulty to distinguish the monoisotopic masses $\ex$ from the MS spectrum $\ey$, especially when different isotopic peaks are present with different intensities in the same mass region.



%
\begin{figure*}[htb]
\centering
\includegraphics[width=8.5cm, height=3.5cm]{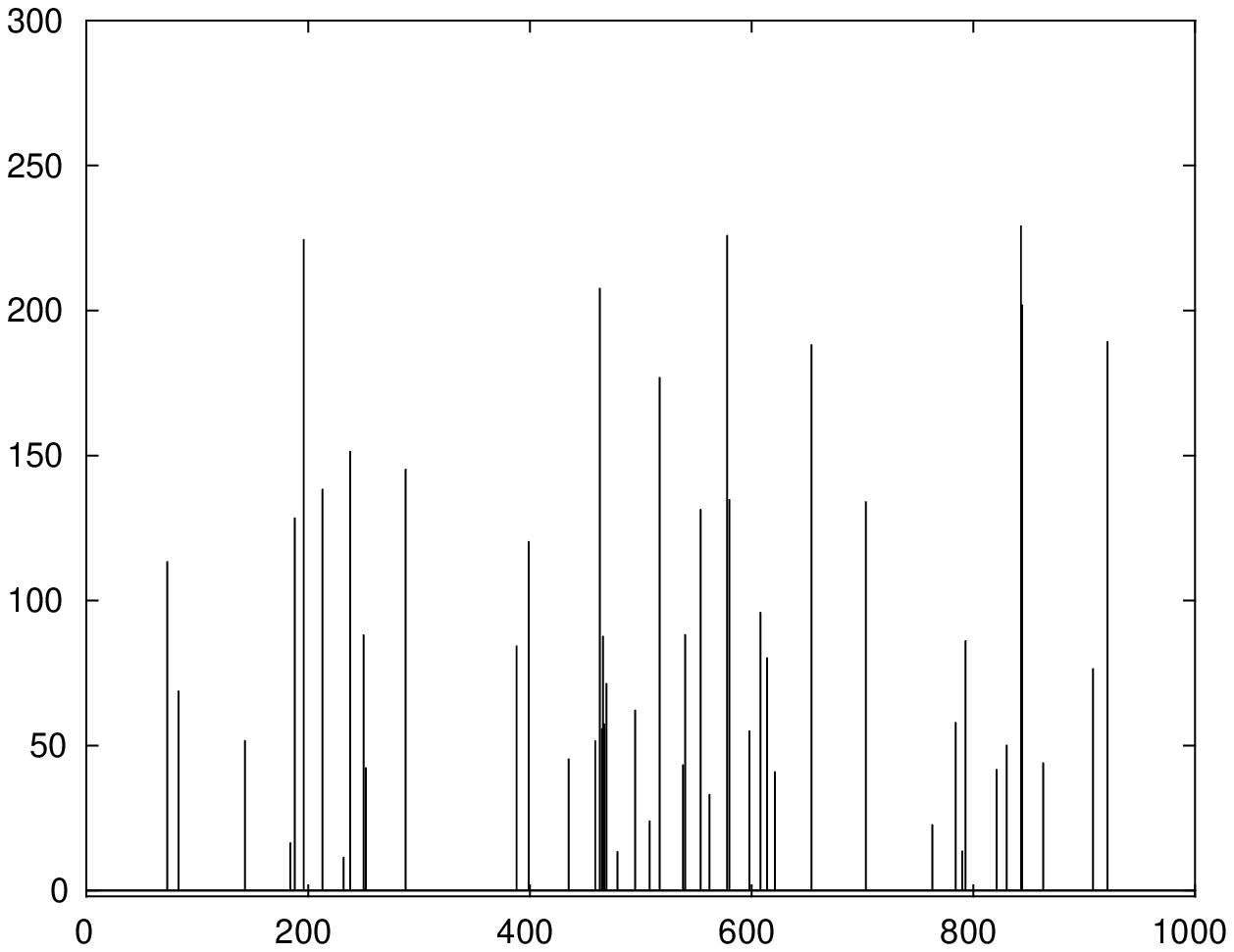}\hspace*{0.6cm}
\includegraphics[width=8.5cm, height=3.5cm]{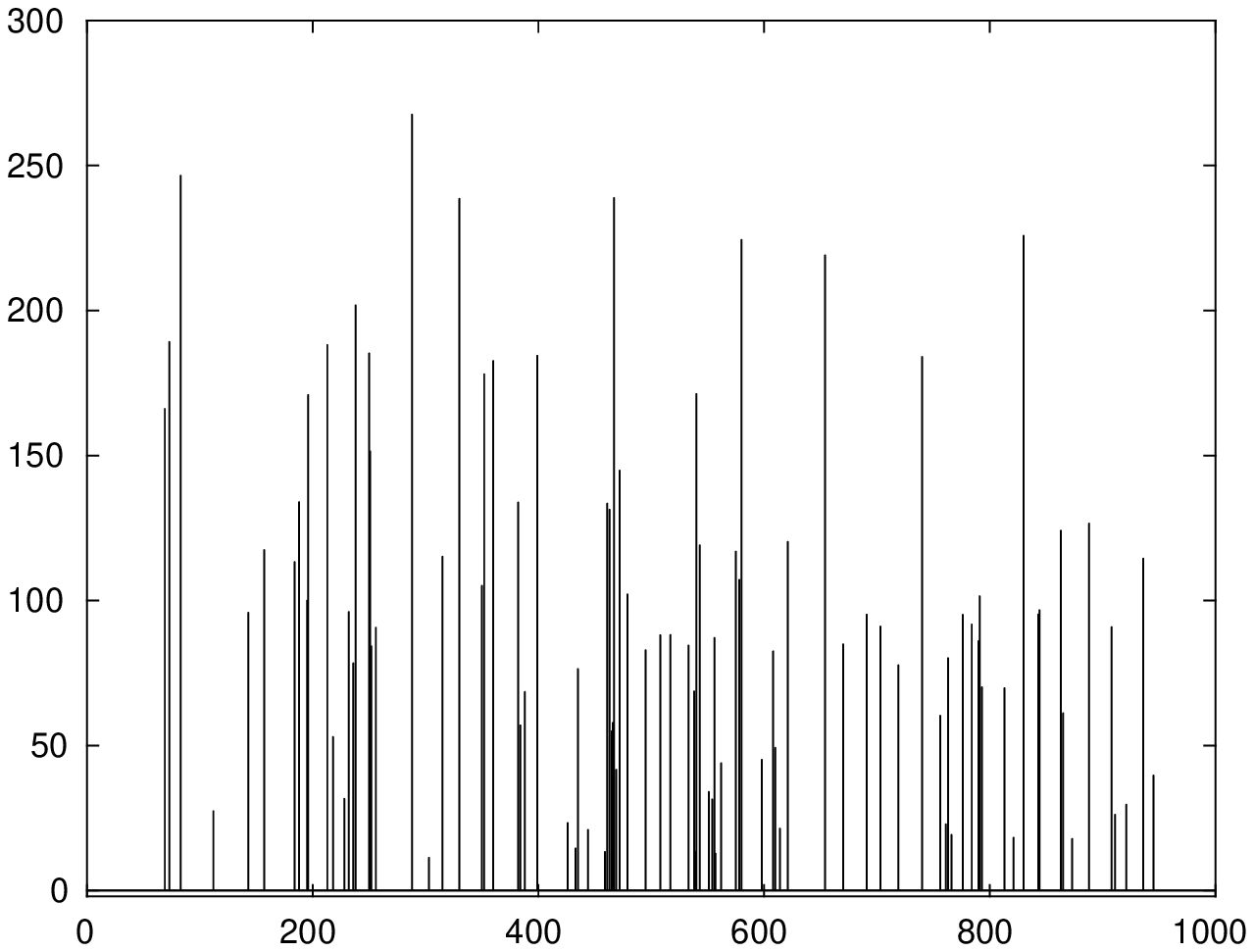}
\end{figure*}
\begin{figure*}[h!]
\centering
\hspace*{0.2cm}
\includegraphics[width=8.3cm, height=3.4cm]{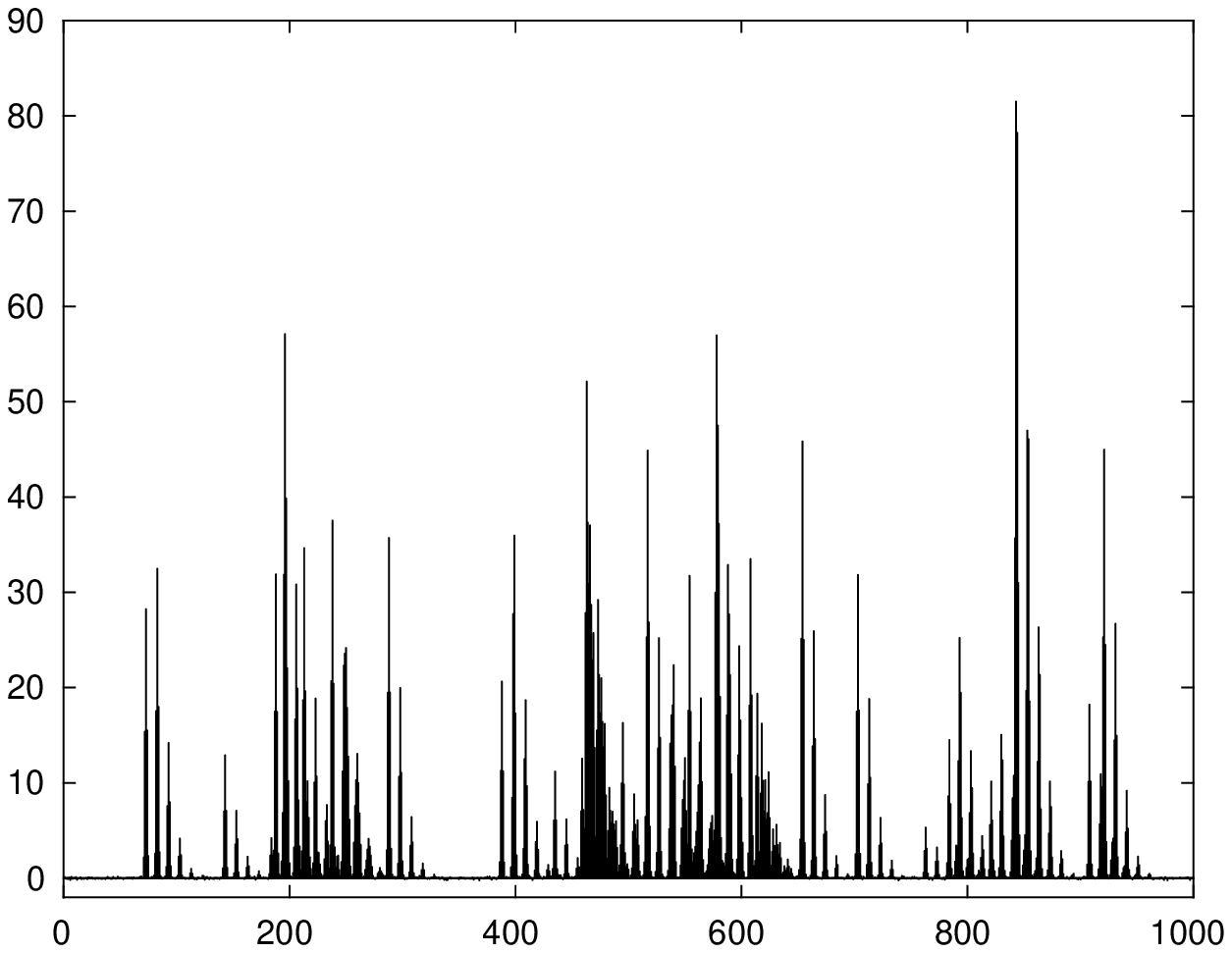}\hspace*{0.8cm}
\includegraphics[width=8.3cm, height=3.4cm]{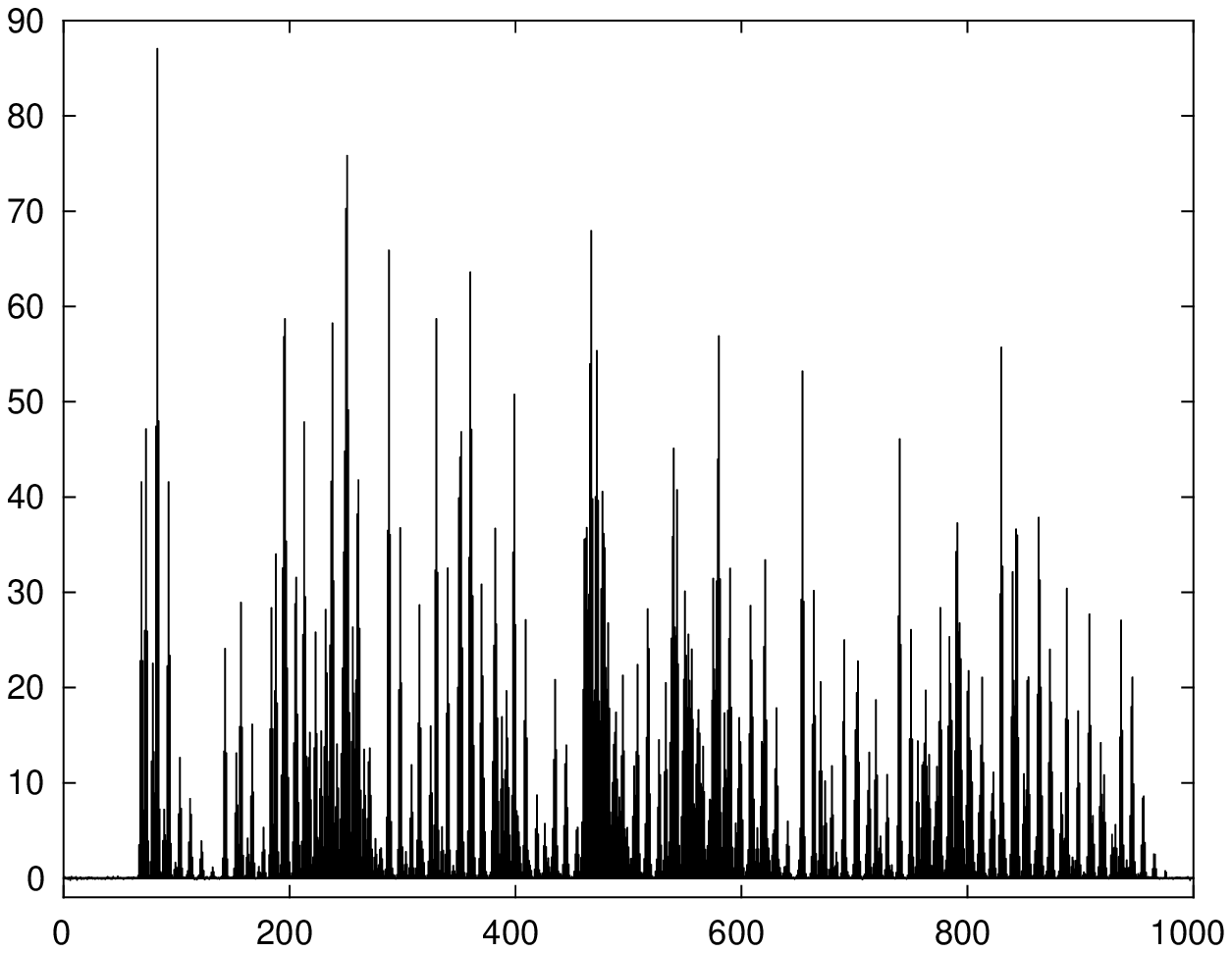}
\caption{Original sparse signals and associated MS spectra for dataset A (left, $N=1000$, $P=48$) and dataset B (right, $N=1000$, $P=94$), 
top: synthetic data, bottom: noisy MS spectra ($\sigma=0.1\%$ of the MS spectrum maximal amplitude).}
\label{Fig:sigAandB}
\end{figure*}

\subsection{Numerical results}

\subsubsection{Comparison of sparse penalties}

Tables \ref{tab:snrlplqA} and \ref{tab:snrlplqB} show the quality reconstruction of signals A and B for different regularization functions and two relative noise levels of the MS spectrum maximal amplitude (\SIlist[round-precision=1]{0.1;0.2}{\percent}), when the SNR, TSNR and sparsity degree are averaged on 10 noise realizations. \cred{It appears that the SPOQ approach yields very good performance for a suitable choice of $p$ and $q$. One can notice that the quality degrades for $p > 1$, especially for small values of $q$. A good compromise seems reached for $p \in \{0.75, 1\}$ and $q \in \{2,3\}$. \cred{Though $\ell_0$, $\ell_1$, \cred{SCAD} and \celo regularization functions ensure a good TSNR, SPOQ shows its clear superiority in all the experiments on this metric. In terms of SNR, it is again superior to all competitors, except on one configuration when SCAD is slightly better. However, in that case, SPOQ is able to provide a better TSNR, showing a better estimation of the sought support. This is confirmed by noticing that the estimated sparsity degree with SPOQ is, in all four experiments, the closest to the reference.} Finally, Cauchy and Welsch have slightly lower performance than the other regularization methods, possibly due to the smoothing induced by parameter $\delta$. These results prove that SPOQ can be the most efficient sparse penalty for an appropriate choice of $p$ and $q$. 
} 


\begin{table*}[ht!]
\tiny
\centering 
\renewcommand{\arraystretch}{1.2}
\setlength{\tabcolsep}{0.11cm}
\begin{minipage}[t]{0.49\linewidth}
\sisetup{round-precision=2}
{
\resizebox{\columnwidth}{!}{\begin{tabular}{|c|c|c|c|c|c|c|c|c|c|c|c|} \cline{3-12}
\multicolumn{2}{c|}{}  & \multicolumn{10}{c|}{$\sigma\ldr{\approx}{=}0.081 $ ($0.1 \%$ of the dataset A maximum amplitude)}  \\ \cline{3-12} 
\multicolumn{2}{c|}{}  & \multicolumn{4}{c|}{\cred{SPOQ}} &  \multirow{2}{*}{$\ell_0$} &  \multirow{2}{*}{$\ell_1$} & \cred{SCAD} & Cauchy & Welsch &  \celo \\ \cline{2-6} 
\multicolumn{1}{c|}{} & \backslashbox{$p$}{$q$} & 2 & 3 & 5 & 10 & &  & \cred{$a=2.25,\, \delta=1$} & $\delta=100$  & $\delta=2$ & $\delta=0.5$ \\ \hline
\multirow{10}{*}{\rotatebox[origin=c]{90}{\centering SNR}} 
 & 0.05   & \textbf{\num{52.1967}} &  \textbf{\num{52.1967}}  & \textbf{\num{52.1967}} &  \textbf{\num{52.1967}} &
\multirow{10}{*}{\num{42.8055}}  & \multirow{10}{*}{\num{47.8471}}  & \multirow{10}{*}{\cred{\num{50.4484}}} & \multirow{10}{*}{\num{41.1964}} & \multirow{10}{*}{\num{30.1605}} & \multirow{10}{*}{\num{37.6823}} \\ \cline{2-6} 
 & 0.1  & \textbf{\num{52.1967}} &  \textbf{\num{52.1967}}  & \textbf{\num{52.1967}} &  \textbf{\num{52.1967}} & & & & & &  \\ \cline{2-6}
 & 0.15 & \textbf{\num{52.1967}} &  \textbf{\num{52.1967}}  & \textbf{\num{52.1967}} &  \textbf{\num{52.1967}} & & & & & & \\ \cline{2-6}
 & 0.2  & \textbf{\num{52.1967}} &  \textbf{\num{52.1967}}  & \textbf{\num{52.1967}} &  \textbf{\num{52.1967}} & & & & & &  \\ \cline{2-6}
 & 0.25 & \textbf{\num{52.1967}} &  \textbf{\num{52.1967}}  & \textbf{\num{52.1967}} &  \textbf{\num{52.1967}} & & & & & & \\ \cline{2-6}
 & 0.5  & \textbf{\num{52.1967}} &  \textbf{\num{52.1967}}  & \textbf{\num{52.1967}} &  \textbf{\num{52.1967}} & & & & & & \\ \cline{2-6}
 & 0.75 & \textbf{\num{52.1967}} &  \textbf{\num{52.1967}}  & \textbf{\num{52.1967}} &  \textbf{\num{52.1967}} & & & & & & \\ \cline{2-6}
 & 1    & \num{50.4025} &  \num{47.6285} &  \num{46.1481} &  \num{44.8749} & & & & & & \\ \cline{2-6}
 & 1.25 & \num{29.8757} &  \num{32.1398} &  \num{34.8351} &  \num{38.4300} & & & & & & \\ \cline{2-6}
 & 1.5  & \num{10.0505} &  \num{34.1361} &  \num{38.1895} &  \num{41.2677} & & & & & & \\ \cline{2-12}
 \hline \hline

\multirow{10}{*}{\rotatebox[origin=c]{90}{\centering TSNR}}
 & 0.05 &  \textbf{\num{52.1967}}  &  \textbf{\num{52.1967}}  &  \textbf{\num{52.1967}}  &   \textbf{\num{52.1967}} & 
\multirow{10}{*}{\num{47.8164}}  & \multirow{10}{*}{\num{49.8708}} & \multirow{10}{*}{\cred{\num{50.4489}}} & \multirow{10}{*}{\num{44.9513}} & \multirow{10}{*}{\num{42.1260}} & \multirow{10}{*}{\num{46.5745}} \\ \cline{2-6} 
 & 0.1  & \textbf{\num{52.1967}} &  \textbf{\num{52.1967}}  & \textbf{\num{52.1967}} &  \textbf{\num{52.1967}} & & & & & & \\ \cline{2-6}
 & 0.15 & \textbf{\num{52.1967}} &  \textbf{\num{52.1967}}  & \textbf{\num{52.1967}} &  \textbf{\num{52.1967}} & & & & & & \\ \cline{2-6}
 & 0.2  & \textbf{\num{52.1967}} &  \textbf{\num{52.1967}}  & \textbf{\num{52.1967}} &  \textbf{\num{52.1967}} & & & & & & \\ \cline{2-6}
 & 0.25 & \textbf{\num{52.1967}} &  \textbf{\num{52.1967}}  & \textbf{\num{52.1967}} &  \textbf{\num{52.1967}} & & & & & & \\ \cline{2-6}
 & 0.5  & \textbf{\num{52.1967}} &  \textbf{\num{52.1967}} & \textbf{\num{52.1967}} &  \textbf{\num{52.1967}} & & & & & & \\ \cline{2-6}
 & 0.75 & \textbf{\num{52.1967}} &  \textbf{\num{52.1967}}  & \textbf{\num{52.1967}} &  \textbf{\num{52.1967}} & & & & & & \\ \cline{2-6}
 & 1   & \num{51.1751} &  \num{49.6989}  & \num{48.8950} &  \num{47.8519} & & & & & & \\ \cline{2-6}
 & 1.25 & \num{38.8529} &  \num{40.8117}  & \num{42.7315} &  \num{44.1649} & & & & & & \\ \cline{2-6}
 & 1.5  & \num{35.8070}  & \num{41.2168}  & \num{43.4731} &  \num{45.4011} & & & & & & \\ \cline{2-12}
 \hline \hline
    
\multirow{10}{*}{\rotatebox[origin=c]{90}{\centering Sparsity}}
 & 0.05 & \textbf{48} & \textbf{48} & \textbf{48} & \textbf{48} & \multirow{10}{*}{236}  & \multirow{10}{*}{77}  & \multirow{10}{*}{\cred{\textbf{\num{48}}}} & \multirow{10}{*}{248}  &\multirow{10}{*}{545}   & \multirow{10}{*}{435}   \\ \cline{2-6}
 & 0.1  & \textbf{48} & \textbf{48} & \textbf{48} & \textbf{48} & & & & & &   \\ \cline{2-6}
 & 0.15 & \textbf{48} & \textbf{48} & \textbf{48} & \textbf{48} & & & & & &   \\ \cline{2-6}
 & 0.2  & \textbf{48} & \textbf{48} & \textbf{48} & \textbf{48} & & & & & &   \\ \cline{2-6}
 & 0.25 & \textbf{48} & \textbf{48} & \textbf{48} & \textbf{48} & & & & & &   \\ \cline{2-6}
 & 0.5  & \textbf{48} & \textbf{48} & \textbf{48} & \textbf{48} & & & & & &  \\ \cline{2-6}
 & 0.75 & \textbf{48} & \textbf{48} & \textbf{48} & \textbf{48} & & & & & & \\ \cline{2-6}
 & 1   & 49 & 59 & 76 & 109 & & & & & &  \\ \cline{2-6}
 & 1.25 & 502 & 457 & 381 & 318 & & & & & &  \\ \cline{2-6}
 & 1.5  & 904 & 396 & 309 & 267 & & & & & &  \\ \cline{2-12}
 \hline
\end{tabular}}
}
\end{minipage}
\hfill\hfill
\begin{minipage}[t]{0.49\linewidth}
\sisetup{round-precision=2}
{
\resizebox{\columnwidth}{!}{\begin{tabular}{|c|c|c|c|c|c|c|c|c|c|c|c|} \cline{3-12}
\multicolumn{2}{c|}{}  & \multicolumn{10}{c|}{$\sigma \approx 0.163$ ($0.2\%$ of the dataset A maximum amplitude)}  \\ \cline{3-12} 
\multicolumn{2}{c|}{}  & \multicolumn{4}{c|}{\cred{SPOQ}} &  \multirow{2}{*}{$\ell_0$}  & \multirow{2}{*}{$\ell_1$} & \cred{SCAD} & Cauchy & Welsch & \celo \\ \cline{2-6}
\multicolumn{1}{c|}{} & \backslashbox{$p$}{$q$} & 2 & 3 & 5 & 10 & & & \cred{$a=2.75,\; \delta=1.25$} & $\delta=2$  & $\delta=2$ & $\delta=0.5$ \\ \hline

\multirow{10}{*}{\rotatebox[origin=c]{90}{\centering SNR}} 

& 0.05 &  \num{31.8303} &  \num{31.1845} &  \num{32.7764}  & \num{37.2335} & 
\multirow{10}{*}{\num{36.7761}}  & \multirow{10}{*}{\num{41.8415}} & \multirow{10}{*}{\cred{\textbf{\num{43.2003}}}} & \multirow{10}{*}{\num{32.3277}} & \multirow{10}{*}{\num{25.1867}} & \multirow{10}{*}{\num{26.9815}} \\ \cline{2-6}
& 0.1  &  \num{31.8303} &  \num{31.1845}  &\num{34.3790}  & \num{38.6395} & & & & & &  \\ \cline{2-6}
& 0.15 &  \num{31.8278} &  \num{31.1845} &  \num{37.2335}  & \num{38.6649} & & & & & &  \\ \cline{2-6}
& 0.2  &  \num{31.8356} &  \num{31.1845}  & \num{38.6395} &  \num{38.6649} & & & & & &  \\ \cline{2-6}
& 0.25 &  \num{31.8494} &  \num{31.1845} &  \num{38.6395} &  \num{38.6649} & & & & & &  \\ \cline{2-6}
& 0.5  &  \num{33.4494} &  \num{38.6649}  & \num{39.9892} &  \num{39.9892} & & & & & &  \\ \cline{2-6}
& 0.75 &  \num{39.6620} &  \num{39.6787} &  \num{40.0177} &  \num{40.0184} & & & & & &  \\ \cline{2-6}
& 1   &  \num{42.7212} &  \num{40.3245} &  \num{40.0551} & \num{38.6998} & & & & & &  \\ \cline{2-6}
& 1.25 &  \num{25.7828} &  \num{24.1823} &  \num{27.9907}  & \num{31.8170} & & & & & & \\ \cline{2-6}
& 1.5  &  \num{-4.7756} &  \num{23.0394}  & \num{30.3910} &  \num{34.9471} & & & & & & \\ \cline{2-12}
\hline \hline
   
\multirow{10}{*}{\rotatebox[origin=c]{90}{\centering TSNR}} 
& 0.05 & \num{34.4757}  & \num{33.3160}  & \num{34.6930} &  \num{38.5110}  &  \multirow{10}{*}{\num{41.7914}} & \multirow{10}{*}{{\num{43.8316}}} & \multirow{10}{*}{\cred{\num{43.7984}}} & \multirow{10}{*}{\num{37.1451}} & \multirow{10}{*}{\num{37.3378}} & \multirow{10}{*}{\num{41.1331}} \\ \cline{2-6}
   
& 0.1  &  \num{34.4757}  & \num{33.3160}  & \num{36.0759} &  \num{39.7377} & & & & & &  \\ \cline{2-6}
& 0.15 &  \num{34.4764} &  \num{33.3160} &  \num{38.5110}  & \num{39.7508} & & & & & &  \\ \cline{2-6}
& 0.2  &  \num{34.4351} &  \num{33.3160} &  \num{39.7377} & \num{ 39.7508} & & & & & &  \\ \cline{2-6}
& 0.25 &  \num{34.4391} &  \num{33.3160} &  \num{39.7377} &  \num{39.7508} & & & & & &  \\ \cline{2-6}
& 0.5  &  \num{35.8299} &  \num{39.7508} &  \num{40.8620} &  \num{40.8620} & & & & & &  \\ \cline{2-6}
& 0.75 &  \num{40.6518} &  \num{40.6471} &  \num{40.8643}  & \num{40.8497} & & & & & & \\ \cline{2-6}
& 1   &  \textbf{\num{43.9063}} & \num{42.6834} &  \num{42.8221}  & \num{41.7592} & & & & & &  \\ \cline{2-6}
& 1.25 &  \num{34.1788} & \num{33.4398} &  \num{36.4043}  &  \num{38.0950} & & & & & &\\ \cline{2-6}
& 1.5  &  \num{28.9497}  & \num{33.9638} &  \num{36.6544}  &  \num{39.3231} & & & & & & \\ \cline{2-12}
\hline \hline

\multirow{10}{*}{\rotatebox[origin=c]{90}{\centering Sparsity}} 
& 0.05 &  59   & 49  &  49 &   \textbf{48} & 
\multirow{10}{*}{236}  & \multirow{10}{*}{75} & \multirow{10}{*}{\cred{\num{49}}} & \multirow{10}{*}{310} & \multirow{10}{*}{496} & \multirow{10}{*}{513} \\ \cline{2-6}
& 0.1  &  59   & 49  &  \textbf{48}  &  \textbf{48} & & & & & & \\ \cline{2-6}
& 0.15 &  59   & 49  &  \textbf{48}  &  \textbf{48} & & & & & & \\ \cline{2-6}
& 0.2  &  59   & 49  &  \textbf{48}  &  \textbf{48} & & & & & & \\ \cline{2-6}
& 0.25 &  59   & 49  &  \textbf{48}  &  \textbf{48} & & & & & & \\ \cline{2-6}
& 0.5  &  59 &  \textbf{48} &  \textbf{48} &  \textbf{48} & & & & & &  \\ \cline{2-6}
& 0.75 &  \textbf{48} &  \textbf{48} &  \textbf{48} &  49 & & & & & &  \\ \cline{2-6}
& 1   &  51 &  64 &  75 & 118 & & & & & &  \\ \cline{2-6}
& 1.25 & 386 & 511 & 410 & 330 & & & & & &  \\ \cline{2-6}
& 1.5  & 957 & 480 & 345 & 273 & & & & & &  \\ \cline{2-12}
\hline
\end{tabular}}
}
\end{minipage}
\caption{Dataset A ($N=1000$, $P=48$): comparison of SNR, TSNR and sparsity degree values averaged on 10 noise realizations using SPOQ with different $p \in ]0,2[$ and $q \in [2, + \infty[$ and some other regularization functions.}
\label{tab:snrlplqA}
\end{table*}

\begin{table*}[ht!]
\tiny
\centering 
\renewcommand{\arraystretch}{1.2}
\setlength{\tabcolsep}{0.11cm}
\begin{minipage}[t]{0.49\linewidth}
\sisetup{round-precision=2}
{
\resizebox{\columnwidth}{!}{\begin{tabular}{|c|c|c|c|c|c|c|c|c|c|c|c|} \cline{3-12}
\multicolumn{2}{c|}{}  & \multicolumn{10}{c|}{$\sigma \approx 0.087$ ($0.1\%$ of the dataset B maximum amplitude)}  \\ \cline{3-12} 
\multicolumn{2}{c|}{}  & \multicolumn{4}{c|}{\cred{SPOQ}} &  \multirow{2}{*}{$\ell_0$} & \multirow{2}{*}{$\ell_1$} & \cred{SCAD} & Cauchy & Welsch & \celo\\ \cline{2-6} 
\multicolumn{1}{c|}{} & \backslashbox{$p$}{$q$} & 2 &	3  & 5 & 10 & & & \cred{$a=3.75,\; \delta=2.75$} & $\delta=100$  & $\delta=5$ & $\delta=0.5$ \\ \hline
\multirow{10}{*}{\rotatebox[origin=c]{90}{\centering SNR}}
 & 0.05 &  \textbf{\num{51.7792}} &  \textbf{\num{51.7792}} & \textbf{\num{51.7792}} &  \textbf{\num{51.7792}} & \multirow{10}{*}{\num{42.6961}} & \multirow{10}{*}{\num{45.3868}} & \multirow{10}{*}{\cred{\num{51.4788}}} & \multirow{10}{*}{\num{39.9649}} & \multirow{10}{*}{\num{37.2220}} & \multirow{10}{*}{\num{39.1707}}\\ \cline{2-6}
& 0.1  & \textbf{\num{51.7792}} &  \textbf{\num{51.7792}} & \textbf{\num{51.7792}} &  \textbf{\num{51.7792}} & & & & & & \\ \cline{2-6}
& 0.15 & \textbf{\num{51.7792}} &  \textbf{\num{51.7792}} & \textbf{\num{51.7792}} &  \textbf{\num{51.7792}} & & & & & &  \\ \cline{2-6}
& 0.2  & \textbf{\num{51.7792}} &  \textbf{\num{51.7792}} & \textbf{\num{51.7792}} &  \textbf{\num{51.7792}} & & & & & &  \\ \cline{2-6}
& 0.25 & \textbf{\num{51.7792}} &  \textbf{\num{51.7792}} & \textbf{\num{51.7792}} &  \textbf{\num{51.7792}} & & & & & & \\ \cline{2-6}
& 0.5  & \textbf{\num{51.7792}} &  \textbf{\num{51.7792}} & \textbf{\num{51.7792}} &  \textbf{\num{51.7792}} & & & & & & \\ \cline{2-6}
& 0.75 & \textbf{\num{51.7792}} &  \textbf{\num{51.7792}} & \textbf{\num{51.7792}} &  \textbf{\num{51.7792}} & & & & & & \\ \cline{2-6}
& 1    & \num{47.8159} &  \num{46.0837} &  \num{44.1265} &  \num{44.3465} & & & & & & \\ \cline{2-6}
& 1.25 & \num{32.2133} &  \num{31.5044}  & \num{34.8987}  & \num{38.6441} & & & & & & \\ \cline{2-6}
& 1.5  & \num{10.2535} &  \num{16.7105} &  \num{36.4629}  & \num{40.3559} & & & & & & \\ \cline{2-6}
\hline \hline

\multirow{10}{*}{\rotatebox[origin=c]{90}{\centering TSNR}} 
& 0.05 &  \textbf{\num{51.7792}} &  \textbf{\num{51.7792}} & \textbf{\num{51.7792}} &  \textbf{\num{51.7792}} & 
\multirow{10}{*}{\num{46.5223}} & \multirow{10}{*}{\num{47.8926}} & \multirow{10}{*}{\cred{\num{51.4791}}} & \multirow{10}{*}{\num{43.0911}} & \multirow{10}{*}{\num{44.7883}} & \multirow{10}{*}{\num{46.6922}}\\ \cline{2-6}

& 0.1  &  \textbf{\num{51.7792}} &  \textbf{\num{51.7792}} & \textbf{\num{51.7792}} &  \textbf{\num{51.7792}} & & & & & &  \\ \cline{2-6}

& 0.15 &  \textbf{\num{51.7792}} &  \textbf{\num{51.7792}} & \textbf{\num{51.7792}} &  \textbf{\num{51.7792}} & & & & & &  \\ \cline{2-6}

& 0.2  &  \textbf{\num{51.7792}} &  \textbf{\num{51.7792}} & \textbf{\num{51.7792}} &  \textbf{\num{51.7792}} & & & & & & \\ \cline{2-6}

& 0.25 &  \textbf{\num{51.7792}} &  \textbf{\num{51.7792}} & \textbf{\num{51.7792}} &  \textbf{\num{51.7792}} & & & & & & \\ \cline{2-6}

& 0.5  &  \textbf{\num{51.7792}} &  \textbf{\num{51.7792}} & \textbf{\num{51.7792}} &  \textbf{\num{51.7792}} & & & & & & \\ \cline{2-6}

& 0.75 &  \textbf{\num{51.7792}} &  \textbf{\num{51.7792}} & \textbf{\num{51.7792}} &  \textbf{\num{51.7792}} & & & & & & \\ \cline{2-6}

& 1    &  \num{49.1673} &  \num{48.2499}  &  \num{46.6108}  & \num{46.8162} & & & & & & \\ \cline{2-6}
& 1.25 &  \num{38.6805} &  \num{38.8708}  &  \num{40.6140}  & \num{43.0939} & & & & & & \\ \cline{2-6}
& 1.5  &  \num{35.2899} &  \num{35.9820}  &  \num{41.6147}  & \num{43.8278} & & & & & & \\ \cline{2-6}
    
\hline \hline
\multirow{10}{*}{\rotatebox[origin=c]{90}{\centering Sparsity}} 
& 0.05 & \textbf{94} & \textbf{94}  &  \textbf{94} &  \textbf{94}  & \multirow{10}{*}{297} & \multirow{10}{*}{157} &  \multirow{10}{*}{\cred{\textbf{\num{94}}}}  & \multirow{10}{*}{330} & \multirow{10}{*}{413} & \multirow{10}{*}{484}\\ \cline{2-6}
& 0.1  & \textbf{94} & \textbf{94}  &  \textbf{94} &  \textbf{94 }& & & & & & \\ \cline{2-6}
& 0.15 & \textbf{94} &\textbf{94}  &  \textbf{94} &  \textbf{94} & & & & & &  \\ \cline{2-6}
& 0.2  & \textbf{94} & \textbf{94}  &  \textbf{94} &  \textbf{94} & & & & & &  \\ \cline{2-6}
& 0.25 & \textbf{94}  & \textbf{94} &  \textbf{94} &  \textbf{94} & & & & & & \\ \cline{2-6}
& 0.5  & \textbf{94}  & \textbf{94} &  \textbf{94 }&  \textbf{94} & & & & & & \\ \cline{2-6}
& 0.75 & \textbf{94}  & \textbf{94} &  \textbf{94} &  \textbf{94} & & & & & & \\ \cline{2-6}
& 1    & 100 & 121 &  169 & 206 & & & & & & \\ \cline{2-6}
& 1.25 & 454 & 567 &  471 & 385 & & & & & & \\ \cline{2-6}
& 1.5  & 904 & 834 &  433 & 345 & & & & & & \\ \cline{1-12}
\end{tabular}}
}
\end{minipage}
\hfill\hfill
\begin{minipage}[t]{0.49\linewidth}
\sisetup{round-precision=2}
{
\resizebox{\columnwidth}{!}{\begin{tabular}{|c|c|c|c|c|c|c|c|c|c|c|c|} \cline{3-12}
\multicolumn{2}{c|}{}  & \multicolumn{10}{c|}{$\sigma \approx 0.174$ ($0.2 \%$ of the dataset B maximum amplitude)}  \\ \cline{3-12} 
\multicolumn{2}{c|}{}  & \multicolumn{4}{c|}{\cred{SPOQ}} &  \multirow{2}{*}{$\ell_0$}  & \multirow{2}{*}{$\ell_1$} & \cred{SCAD} &Cauchy & Welsch & \celo\\ \cline{2-6} 
\multicolumn{1}{c|}{} & \backslashbox{$p$}{$q$} & 2 &	3  & 5 & 10 &  & & \cred{$a=2.75,\; \delta=1.25$} & $\delta=90$  & $\delta=5$ & $\delta=0.5$ \\ \hline
\multirow{10}{*}{\rotatebox[origin=c]{90}{\centering SNR}} 

& 0.05 &  \num{33.6774} &  \num{33.6774}  & \num{31.9163}  & \num{30.9107} & 
\multirow{10}{*}{\num{45.4554}} & \multirow{10}{*}{39.1028} & \multirow{10}{*}{\cred{\num{42.5086}}} & \multirow{10}{*}{\num{31.6401}} & \multirow{10}{*}{29.1765} & \multirow{10}{*}{\num{30.7932}}\\ \cline{2-6} 
 & 0.1  &  \num{33.6774}  & \num{33.6774}  & \num{32.4214} &  \num{31.5043} & & & & & &  \\ \cline{2-6}
 & 0.15 &  \num{33.8662}  & \num{33.8662}  & \num{32.5359} & \num{31.5566} & & & & & &  \\ \cline{2-6}
 & 0.2  &  \num{33.4058} &  \num{33.3954}  & \num{32.9254} &  \num{32.9888} & & & & & &  \\ \cline{2-6}
 & 0.25 &  \num{33.5835} &  \num{33.9269}  & \num{34.4602} & \num{33.3914} & & & & & &  \\ \cline{2-6}
 & 0.5  &  \num{40.9727} &  \num{40.2355}  & \num{42.4438} & \num{36.7684} & & & & & &  \\ \cline{2-6}
 & 0.75 &  \textbf{\num{45.7586}} &  \textbf{\num{45.7586}}& \textbf{\num{45.7586}} &  \num{41.9264} & & & & & &  \\ \cline{2-6}
 & 1   &   \num{41.9905} & \num{39.6167} & \num{38.6775}  & \num{38.5151} & & & & & &  \\ \cline{2-6}
 & 1.25 &  \num{25.8749} & \num{25.1590} & \num{26.4849}  & \num{32.2250} & & & & & &  \\ \cline{2-6}
 & 1.5  &  \num{0.0346}  & \num{3.2290}  & \num{27.5926} &  \num{33.9418} & & & & & &  \\ \cline{2-6}
\hline \hline

\multirow{10}{*}{\rotatebox[origin=c]{90}{\centering TSNR}} 
& 0.05 & \num{35.3875}  & \num{35.3875}  & \num{33.4864} &  \num{32.6164} & 
\multirow{10}{*}{\textbf{\num{45.4598}}} & \multirow{10}{*}{\num{41.5453}} & \multirow{10}{*}{\cred{\num{43.3415}}} & \multirow{10}{*}{\num{34.8238}} & \multirow{10}{*}{\num{37.6808}} & \multirow{10}{*}{\num{41.5374}}\\ \cline{2-6}

& 0.1  & \num{35.3875} &  \num{35.3875} &  \num{33.5766}  & \num{32.8673} & & & & & &  \\ \cline{2-6}
& 0.15 & \num{35.6548} &  \num{35.6548} &  \num{33.4765}  & \num{32.9279} & & & & & &  \\ \cline{2-6}
& 0.2  & \num{35.0129} &  \num{35.0024} &  \num{33.9797} &  \num{34.1767} & & & & & &  \\ \cline{2-6}
& 0.25 & \num{35.2704} &  \num{35.5057} &  \num{35.1311} & \num{34.5387} & & & & & & \\ \cline{2-6}
& 0.5  & \num{41.8128} &  \num{40.6373} & \num{42.6521}  & \num{38.0371} & & & & & &  \\ \cline{2-6}
& 0.75 & \textbf{\num{45.7586}} &  \textbf{\num{45.7586}} &  \textbf{\num{45.7586}} &  \num{42.6635} & & & & & &  \\ \cline{2-6}
& 1   & \num{43.0874} &  \num{41.7618} &  \num{41.2854} & \num{41.0052} & & & & & &  \\ \cline{2-6}
& 1.25 & \num{32.3486} &  \num{32.6567} & \num{33.6302} & \num{36.9011} & & & & & &  \\ \cline{2-6}
& 1.5  & \num{28.8606} &  \num{29.5998} & \num{33.9029} &  \num{37.6653} & & & & & &  \\ \cline{2-6}

\hline \hline
\multirow{10}{*}{\rotatebox[origin=c]{90}{\centering Sparsity}} 
& 0.05 & 95 & 95 & \textbf{94} & \textbf{94} & 
\multirow{10}{*}{\textbf{94}} & \multirow{10}{*}{155} & \multirow{10}{*}{\cred{\num{96}}} & \multirow{10}{*}{330} & \multirow{10}{*}{472} & \multirow{10}{*}{531}\\ \cline{2-6}
& 0.1  & 95 & 95 & \textbf{94} & \textbf{94}  & & & & & &  \\ \cline{2-6}
& 0.15 & 95 & 95 & \textbf{94} & \textbf{94}  & & & & & &  \\ \cline{2-6}
& 0.2  & \textbf{94} & \textbf{94} & \textbf{94} & 93  & & & & & &  \\ \cline{2-6}
& 0.25 & \textbf{94}  & \textbf{94} &  93 &  93 & & & & & & \\ \cline{2-6}
& 0.5  & \textbf{94}  & \textbf{94} &  \textbf{94} &  \textbf{94} & & & & & & \\ \cline{2-6}
& 0.75 & \textbf{94}  & \textbf{94} &  \textbf{94} &  \textbf{94} & & & & & & \\ \cline{2-6}
& 1    & 97  & 124 &  176 & 208 & & & & & & \\ \cline{2-6}
& 1.25 & 439 & 559 &  534 & 397 & & & & & & \\ \cline{2-6}
& 1.5  & 955 & 893 &  513 & 355 & & & & & & \\ \cline{1-12}
\end{tabular}}
}
\end{minipage}
\caption{Dataset B ($N=1000$, $P=94$): comparison of SNR, TSNR and sparsity degree values averaged on 10 noise realizations using SPOQ with different $p \in ]0,2[$ and $q \in [2, + \infty[$  and some other regularization functions.}
\label{tab:snrlplqB}
\end{table*}

\subsubsection{Advantage of trust-regions}

Figure \ref{Fig:TestTimexAxB} shows the convergence profile, in terms of SNR evolution, of the trust-region VMFB algorithm \eqref{alg:TRVMFB}, the VMFB algorithm and the FB algorithm, to recover datasets A and B when \cred{$p=0.75$} and $q=2$, for a given noise realization. Let us remind that VMFB algorithm corresponds to \eqref{eq:proxineexact}. Here, we set $\eA_k = \eA_{q,0}(\ex_k)$ and $\gamma_k = 1.9$ for every $k \in \eN$. FB algorithm is obtained by setting $\eA_k = L \, \eI_N$ and $\gamma_k = 1.9$ in \eqref{eq:proxineexact}, where $L$ is the Lipschitz constant given by \eqref{eq:LipSPOQ}. It is worth noting that our trust-region VMFB algorithm converges much faster than the two other variants, which themselves behave here quite similarly. 
\cred{As already stated, the numerical value of the Lipschitz constant $L$ can reach high values, when the smoothing parameters are small. For instance, $L= 1.016 \times 10^{12}$ in the example presented in Figure \ref{Fig:TestTimexAxB}. In practice, one can notice that the potential negative effect of such large constant on the convergence speed is avoided thanks to our proposed trust-region approach. This illustrates the advantage of our local preconditioning scheme.}

\begin{figure}[h]
\centering
\hspace*{-0.5cm}
\begin{tabular}{@{}c@{}}
\includegraphics[width=8cm]{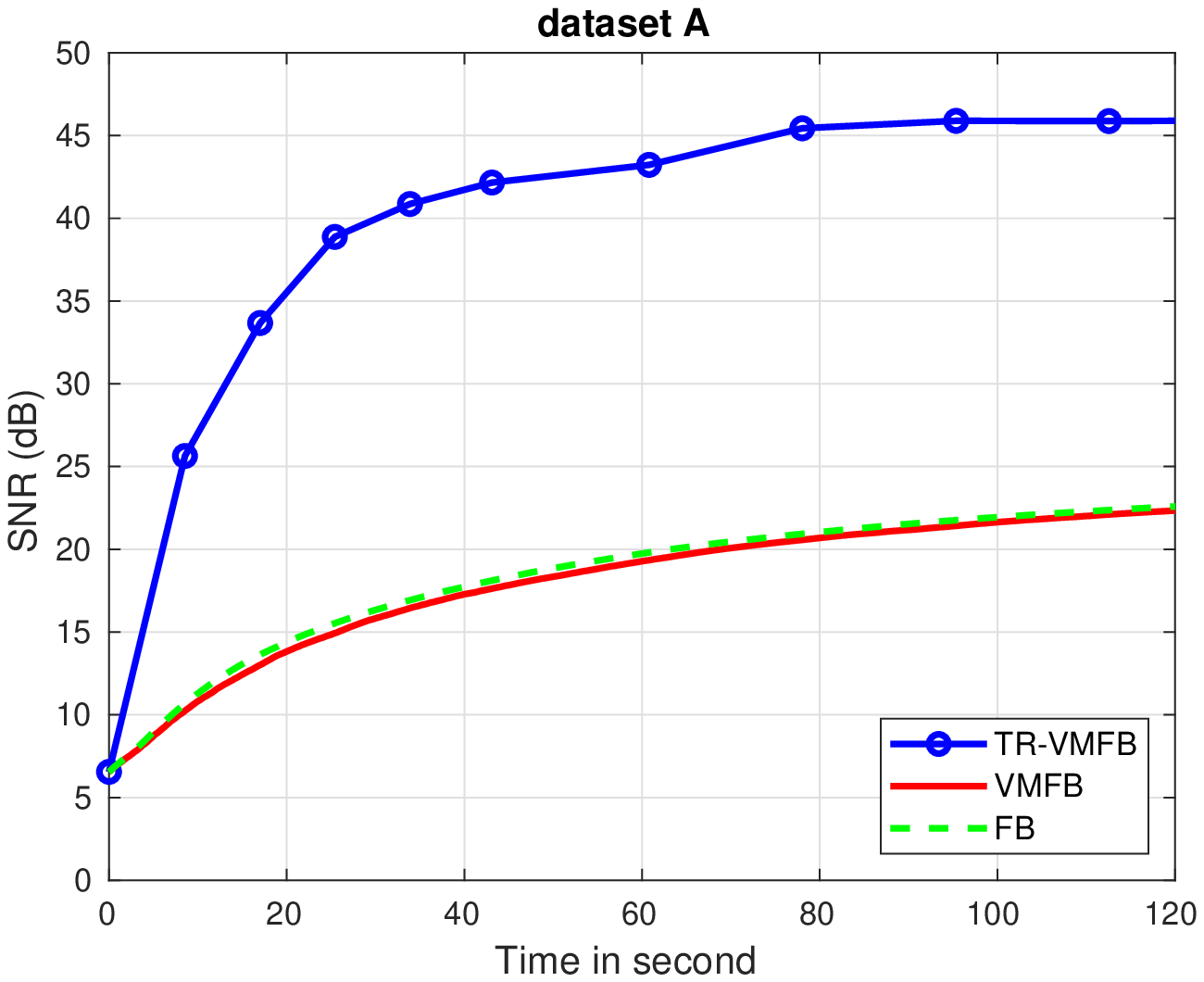} \\
\includegraphics[width=8cm]{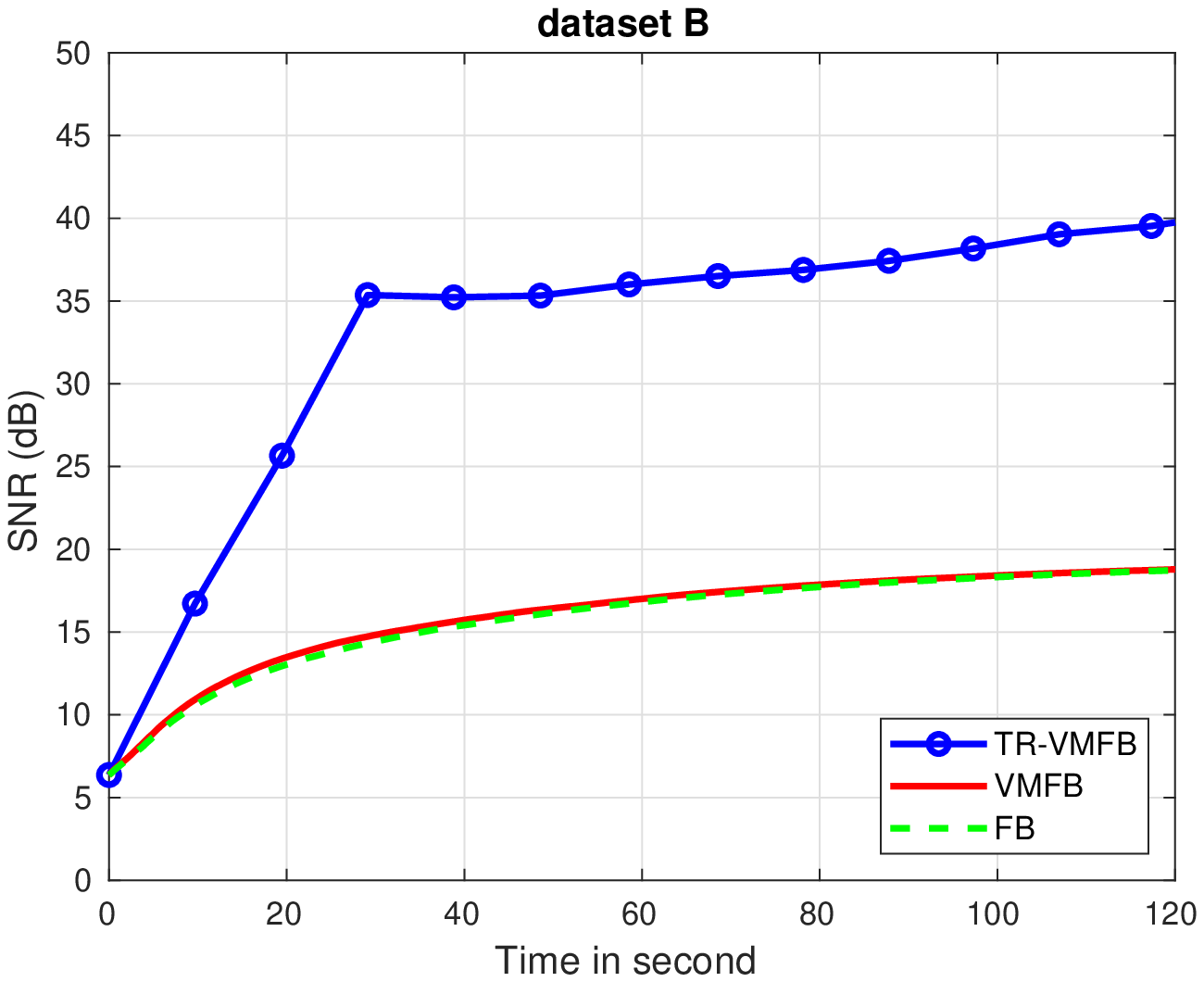} 
\end{tabular}
\caption{SNR evolution along time, for the proposed trust-region VMFB algorithm \ref{alg:TRVMFB}, VMFB algorithm \cite{Chouzenoux_E_2014_j-optim-theory-appl_variable_mfbamsdfcf} and FB algorithm \cred{with $L= 1.016 \times 10^{12}
$}, to process datasets A and B on a given noise realization (relative noise level: \SI[round-precision=1]{0.1}{\percent}).}
\label{Fig:TestTimexAxB}
\end{figure}

\subsubsection{Setting SPOQ parameters}
In all our tests, the smoothing parameters $(\delta, a)$ for Cauchy, Welsch, \cred{SCAD} and \celo penalties were chosen empirically, so as to maximize the final SNR. \eca{Aside, we provide a pairwise sensitivity analysis for the  $\alpha, \beta$ and $\eta$ parameters \eqref{eq:reg_sparse} of SPOQ in Figure \ref{Fig:Testalphabetamu}. We consider dataset A, for the noise level \SI[round-precision=1]{0.1}{\percent}, and the setting $p=0.75$ and $q=2$ as it was observed to lead to the best results in this case. One parameter being fixed, we cover a large span of orders of magnitude for the two others ($\alpha \in [10^{-7}, 10^2]$, $\beta \in [10^{-7}, 10^2]$ and $\eta \in [10^{-7},10^2]$). The first observation is the layered structure of both figures. This is interpreted as the notably weak interdependence of hyperparameters, which is \cred{advantageous}. Secondly, the horizontal red/dark red strip, where the best SNR performance is attained, is relatively large, spanning about one order in magnitude in the tuned parameter. This suggests robustness, with tenuous performance variation through mild parameter imprecision. Thirdly, $\alpha$ seems to have little impact, especially when $\beta$ and $\eta$  are optimized. Note that we did not display the variations for fixed $\alpha$ as we observed that the SNR exhibits non-noticeable value variations. Parameter $\alpha$ essentially controls the $L$-Lipschitz value \eqref{eq:LipSPOQ} and the derivability of $\ell_{p,\alpha}$ at $0$ (see \eqref{eq:lpsmooth}).} 



\begin{figure}[h]
\centering
\begin{tabular}{@{}c@{}}
\includegraphics[width=8cm]{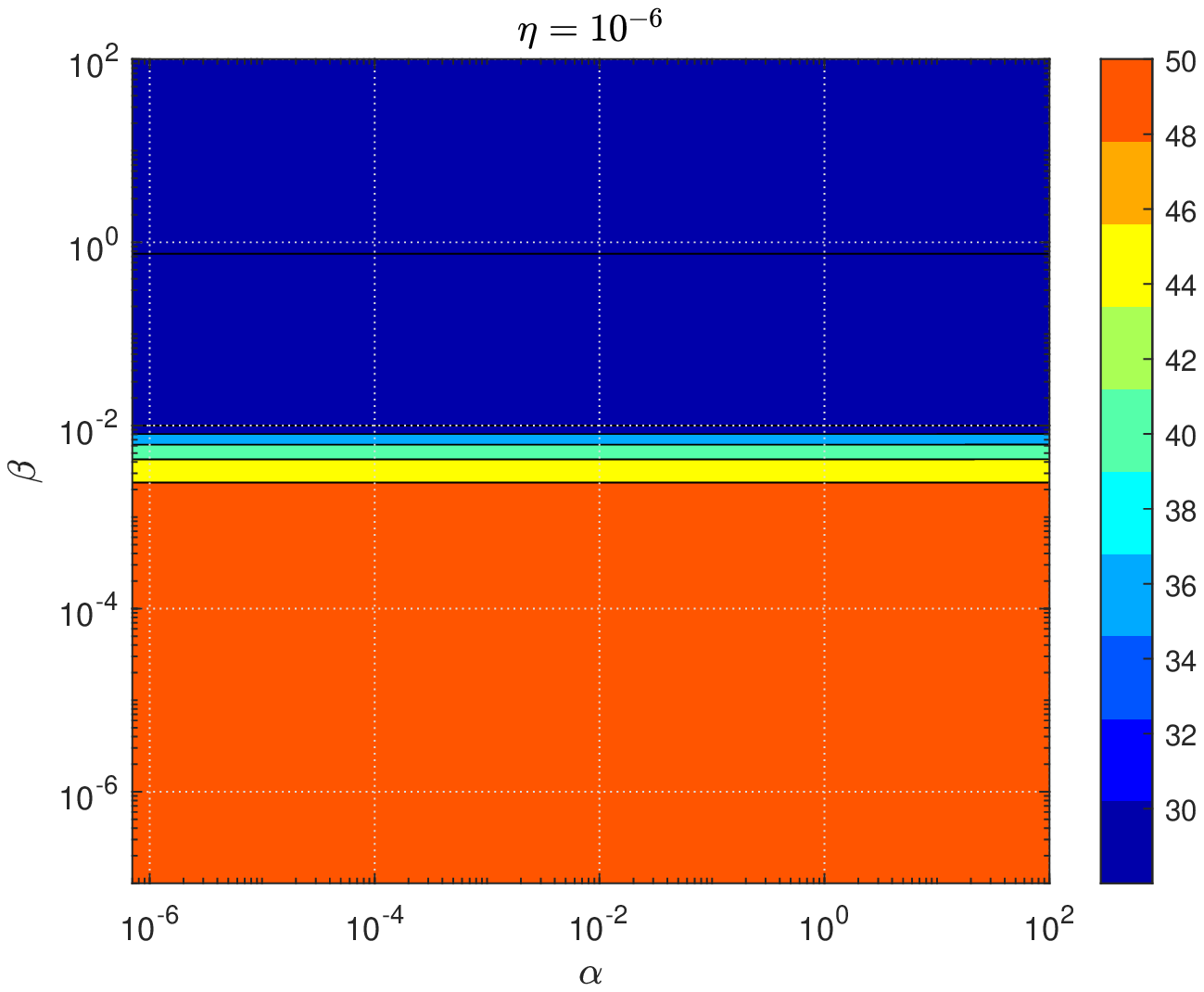}\\
\includegraphics[width=8cm]{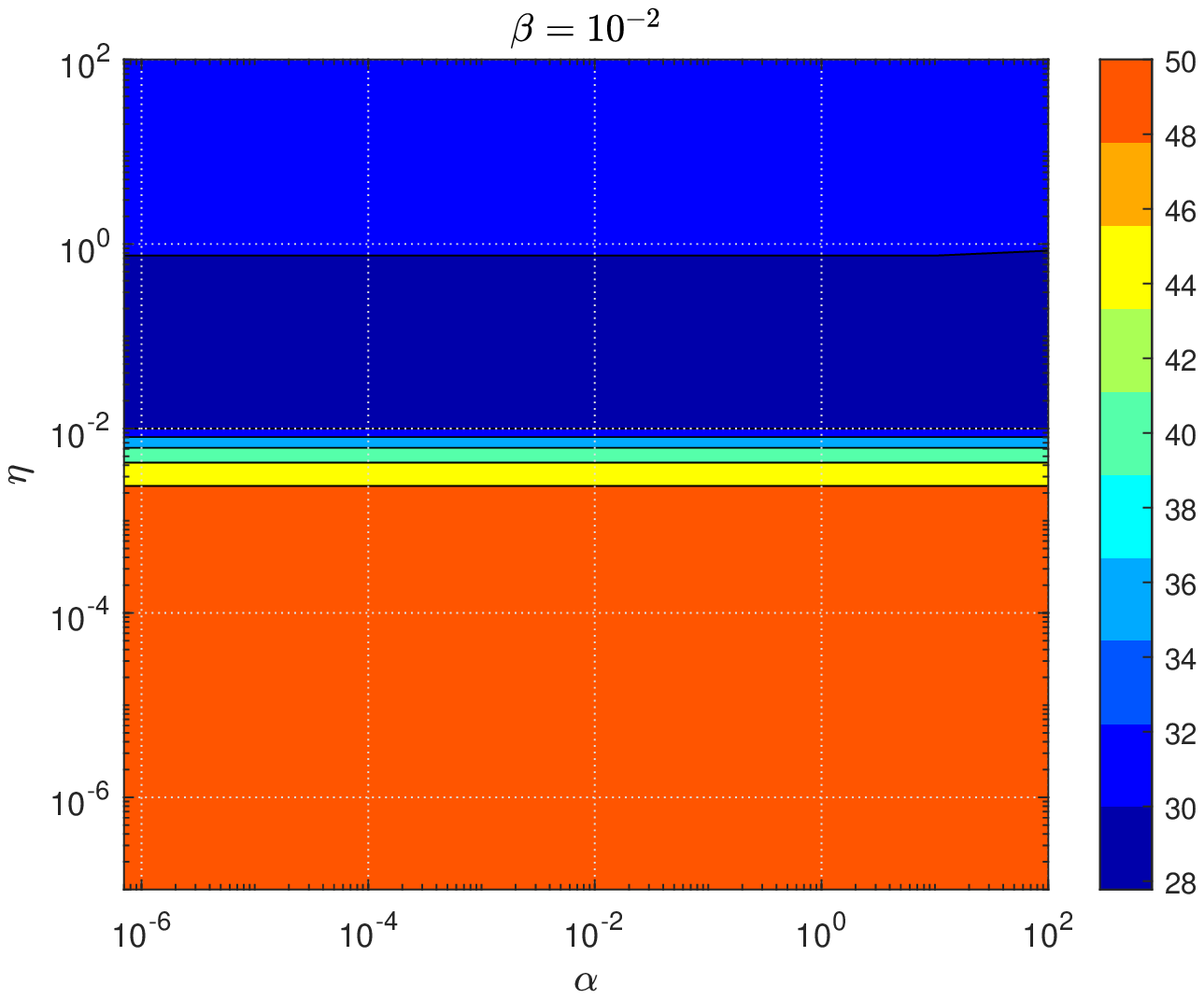}
\end{tabular}
\caption{SNR computed for dataset A ($N=1000$,  $P=48$) using SPOQ regularization with different $\alpha$, $\beta$ and $\eta$ parameters where \eca{$p=0.75$} and $q=2$ (relative noise: \SI[round-precision=1]{0.1}{\percent}).}
\label{Fig:Testalphabetamu}
\end{figure}


\subsubsection{Noise level influence}
Using different penalties ($\ell_0$, $\ell_1$, Cauchy, Welsch, \cred{SCAD}, \celo and two instances of SPOQ), we present  \cred{in Figure \ref{Fig:NoiseLevelTest}} SNR values obtained from datasets A and B reconstruction at different noise levels. \eca{As expected,  SNR for all methods decreases as  noise intensity increases}. Let us remind that the standard deviation $\sigma$ in our case is expressed as a percentage of the MS spectrum maximal amplitude. A noise level greater than $\SI[round-precision=1]{0.1}{\percent}$ corresponds here to a quite high noise level for our datasets, and obviously leads  to a deterioration of reconstruction quality. 
\cred{SPOQ proves its capability to ensure the best quality reconstruction in comparison with others penalties. 
The choice $p=0.75$ and $q=2$ shows its superiority over SOOT (i.e. $p=1$ and $q=2$) for all tested noise levels.}
\cred{Moreover, it is worthy to notice that, on this example, the superiority of SPOQ versus SCAD becomes more significant, as the noise level increases.}
\begin{figure}[h!]
\centering
\begin{tabular}{c}
\includegraphics[width=8cm]{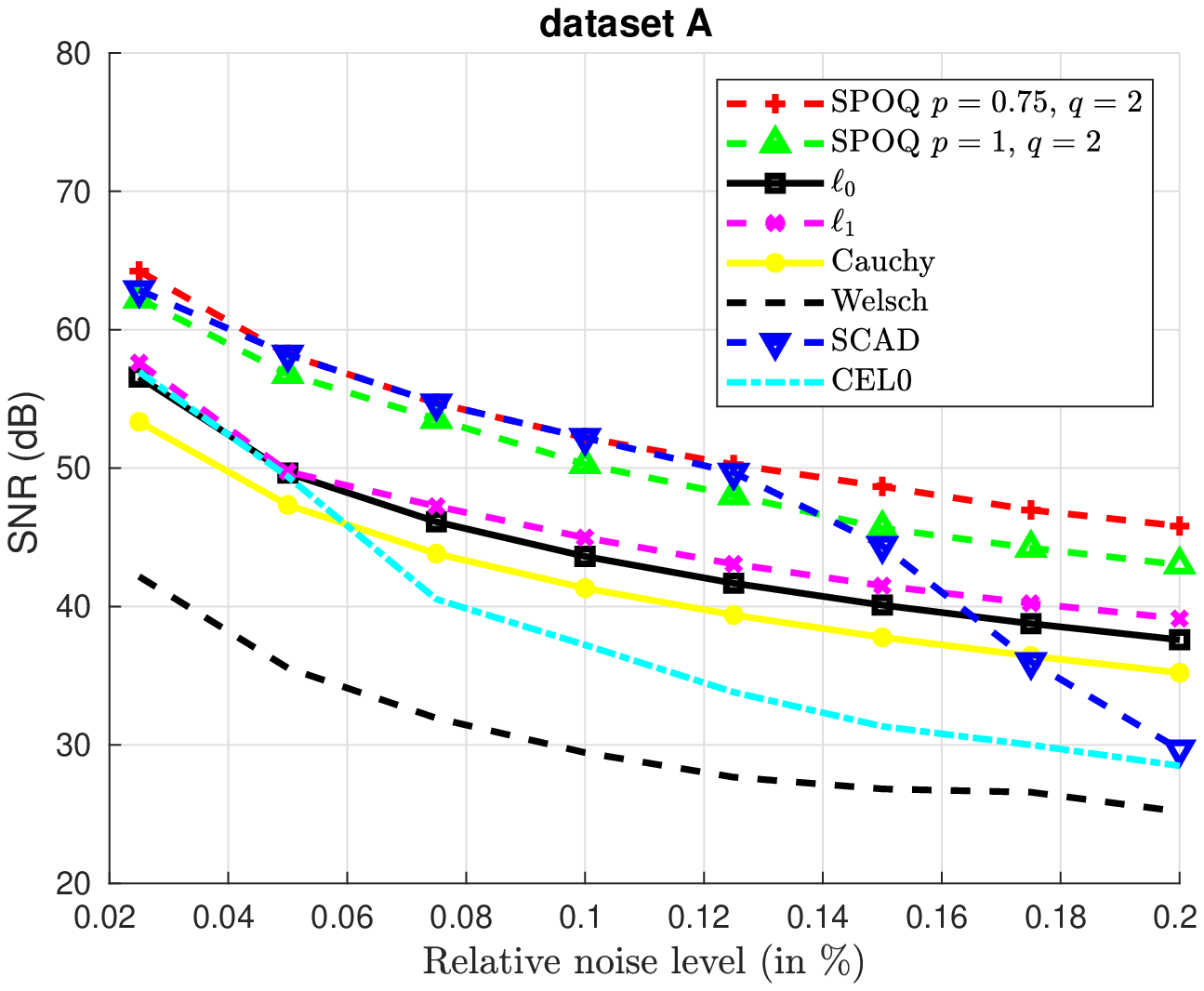}\\
\includegraphics[width=8cm]{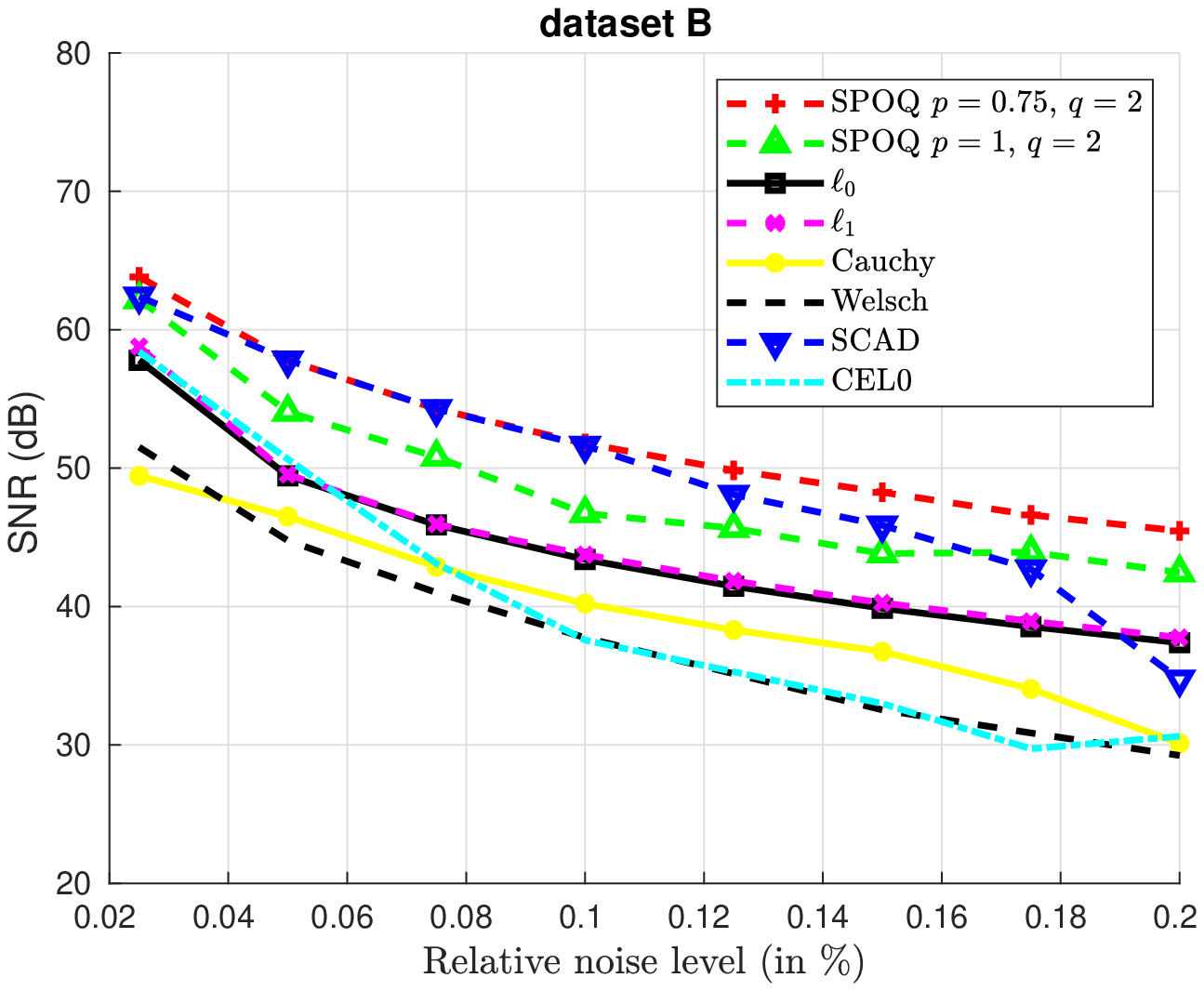}
\end{tabular}
\caption{Influence of noise level ($\sigma$ expressed as a percentage of the MS spectrum maximal amplitude) on quality reconstruction of datasets A (top) and dataset B (bottom) using various penalties: SPOQ \cred{with $p=0.75$ and $q=2$}, SPOQ \cred{with $p=1$ and $q=2$} \cred{ (or SOOT)}, $\ell_0$, $\ell_1$, Cauchy,  Welsch, \cred{SCAD} and \celo (SNR values averaged over 10 noise realizations).}
\label{Fig:NoiseLevelTest}
\end{figure}


\subsubsection{Sparsity level influence}
Our final test consists in evaluating the performance of SPOQ penalty for various sparsity degrees. To do so, we tried out different datasets with a fixed size $N=1000$ and different sparsity degrees $P \in \{ 10, 20, 48, 94, 182, 256, 323, 388 \}$, generated in a similar fashion as in our datasets A and B. We make use of the SPOQ penalty with $p \in \{0.25, 0.75, 1\}$ and $q = 2$. Figure \ref{Fig:Testsparsitydegree} presents the evolution of estimated sparsity degree. As we can see, the latter is well estimated when the signal presents a high sparsity level (Figure \ref{Fig:Testsparsitydegree}, case of $p=0.25$ and $q=2$, $p=0.75$ and $q=2)$. However as $P$ increases, the reconstruction quality of SPOQ where $p=1$ and $q=2$ (i.e., SOOT) tends to worsen. \cred{This confirms the interesting flexibility in setting parameter $p$}.

\begin{figure}[h!]
	\centering
		\begin{tabular}{c}
		\includegraphics[width=8cm]{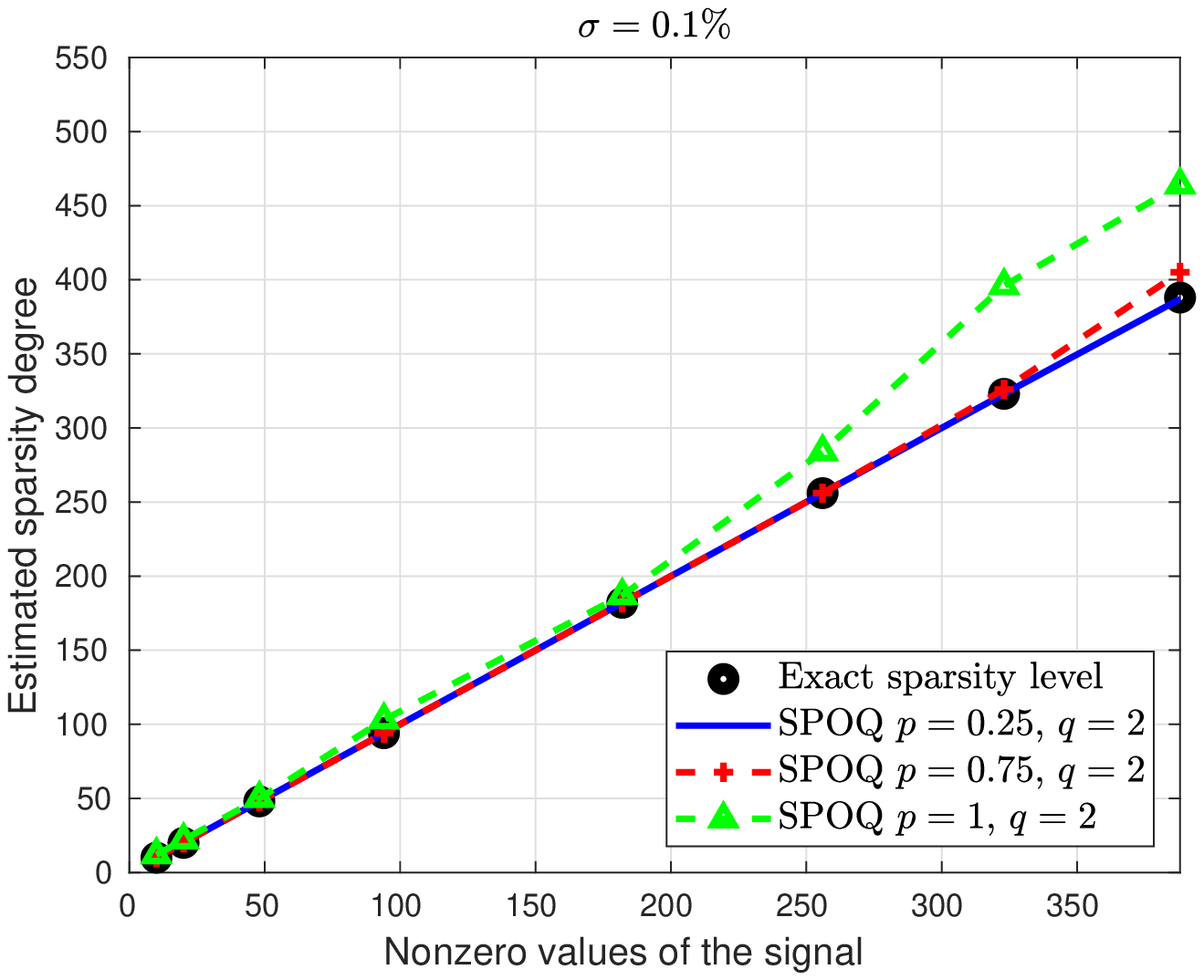}\\
		\includegraphics[width=8cm]{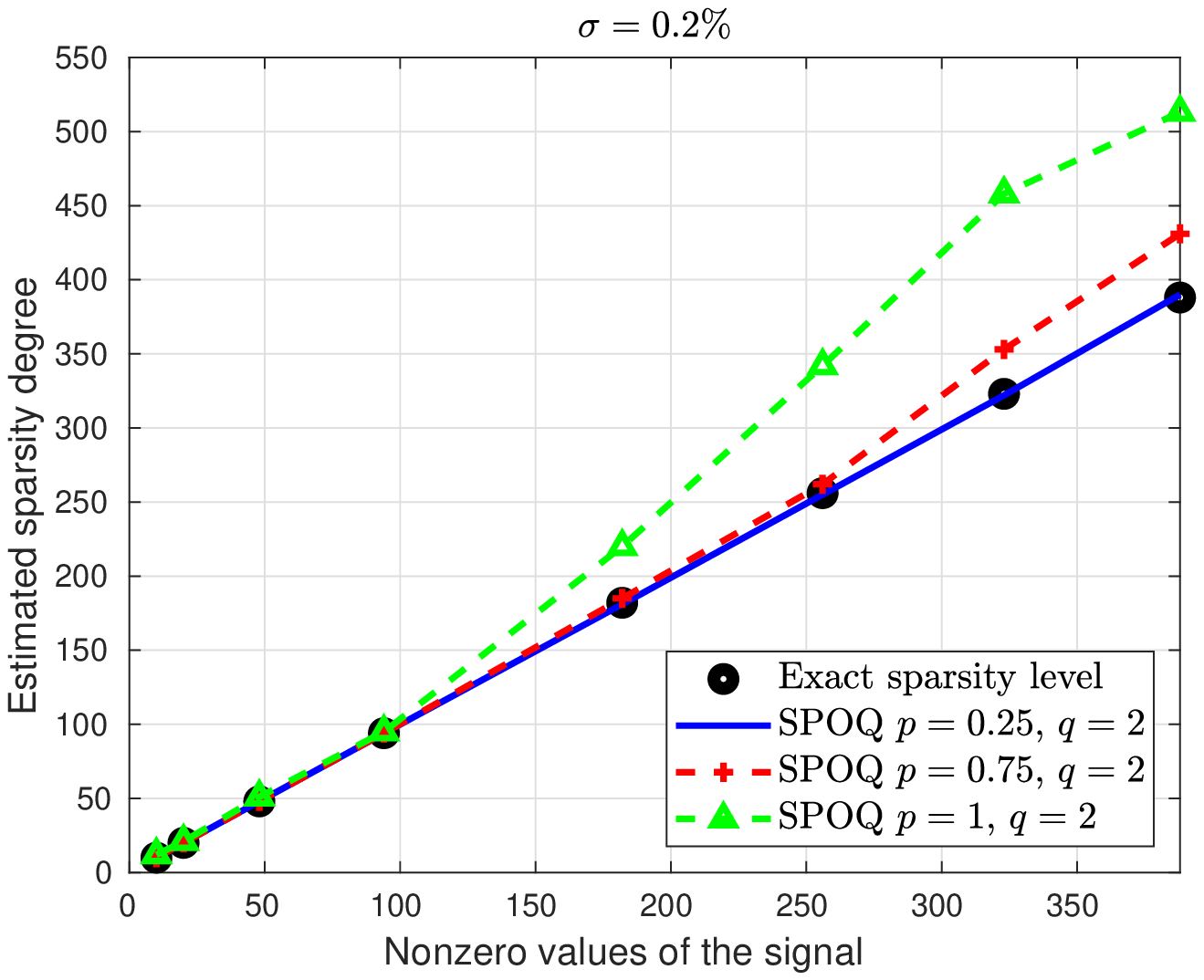}
	\end{tabular}
	\caption{Estimated sparsity degree for different \afa{sparse} signals \afd{sizes} using SPOQ on a single noise realization (relative noise: \SI[round-precision=1]{0.1}{\percent} \cred{(up)} and \SI[round-precision=1]{0.2}{\percent} \cred{(bottom)}).}
	\label{Fig:Testsparsitydegree}
\end{figure}


 

\section{Conclusion}
SPOQ offers scale-invariant penalties, based on  ratios of smoothed quasinorms and norms. These surrogates to the $\ell_0$ count index are non-convex, yet possess Lipschitz regularity, that permits \ldd{to deploy} efficient optimization algorithms based on the majorize-minimize methodology. In particular, we \cred{propose} a novel trust-region approach, that extends the variable metric forward-backward algorithm. On sparse mass-spectrometry peak signals, \cred{SPOQ} outperforms other sparsity penalties for various quality metrics. Moreover,  once the norm exponents are chosen, smoothing hyperparameters are easy to set. Further works include algorithmic acceleration and application to other types of sparse data processing, such as image deconvolution. \cred{SPOQ penalties are finally promising for sparsity estimation,  sparse regression or feature selection}.


%

\appendices
\section{Proof of Proposition \ref{prop:minimiseur}}
\label{Proof0}

First, we have $\nabla \ell_{p,\alpha}^p(\zerob_N) = \nabla \ell_{q,\eta}^q(\zerob_N) = \zerob_N$, so that $\zerob_N$ is a critical point of $\Psi$. 
By using \eqref{e:D2phi1},
\begin{align}
\nabla^2 \Psi_1(\zerob_N) &=  \frac{\alpha^{p-2}}{\beta^p} \eI_N,
\end{align}
with $\eI_N$ identity matrix of $\eR^N$.
If $q>2$, it follows from \eqref{eq_phi-hessian1} that
\begin{align}
\nabla^2  \Psi_2(\zerob_N) 
&= \frac{1}{q}\left( \frac{q(q-1) \DD{((0^{q-2})_{1\leq n \leq N})}}{\eta^q} - \frac{q}{\eta^{2q}}\zerob_{N \times N}  \right)\nonumber\\
&= \zerob_{N \times N},
\end{align}
otherwise, if $q=2$,
\begin{align}
\nabla^2  \Psi_2(\zerob_N) 
&= \frac{1}{2}\left( \frac{2(2-1) \DD{((0^0)_{1 \leq n \leq N})}}{\eta^2} - \frac{2}{\eta^{2q}}\zerob_{N \times N}  \right)\nonumber\\
&= \frac{1}{\eta^2} \eI_N 
\end{align}
since  $0^{0}=1$ by convention. Consequently
\begin{align}
\nabla^2 \Psi_2(\zerob_N)  &= \begin{cases}
\displaystyle \frac{1}{\eta^2} \eI_N & \mbox{if } q = 2, \\
\zerob_{N \times N} & \mbox{elsewhere}.
\end{cases}
\end{align}
According to these results, we deduce that $\nabla^2 \Psi(\zerob_N)$ is a positive definite matrix if ($q = 2$ and $\eta^2 \alpha^{p-2} > \beta^p$) or if $q > 2$. When these conditions are fulfilled, $\zerob_N$ is a local minimizer of $\Psi$.

Let us now show that, under suitable assumptions,
\begin{align}\label{e:psimin0}
(\forall \ex \in \mathbb{R}^{N)}\quad & \Psi(\ex) \ge  \Psi(\zerob_N) \nonumber\\
\Leftrightarrow\quad  & \dfrac{(\ell_{p,\alpha}^p(\ex) + \beta^p)^{1/p}}{\ell_{q,\eta}(\ex) } \ge \frac{\beta}{\eta}
\end{align}
that is 
\begin{multline}\label{e:ineqglobopt}
\left(1+\sum_{n=1}^{N} \frac{\alpha^p}{\beta^p}\Big(\big(\frac{z_{n}}{\alpha^{2}} + 1\big)^{p/2}  - 1\Big)\right)^{2/p}\\
\ge \left( 1 + \sum_{n=1}^N \frac{z_n^{q/2}}{\eta^q} \right)^{2/q}
 \end{multline}
 by setting, for every $n\in\{1,\ldots,N\}$, $z_{n}= x_n^2$. Let $\epsilon \in ]0,+\infty[$. According to the second-order mean value theorem,
 \begin{equation}
 (\forall \upsilon \in [0,\epsilon])\quad (\upsilon+1)^{p/2}-1 \ge \frac{p\upsilon}{2} \Big(1-\frac{2-p}{4}\epsilon\Big).
 \end{equation}
 On the other hand, since $\upsilon \mapsto \big((\upsilon+1)^{p/2}-1\big)/\upsilon^{p/2}$ is an increasing function on $]0,+\infty[$,
 \begin{equation}
 (\forall \upsilon \in [\epsilon,+\infty[)\quad (\upsilon+1)^{p/2}-1 \ge \frac{(\epsilon+1)^{p/2}-1}{\epsilon^{p/2}} \upsilon^{p/2}.
 \end{equation}
 In the following, we will assume that $\epsilon< 4/(2-p)$. Let
 \begin{equation}
 I = \{n \in \{1,\ldots,N\}\mid z_{n}< \epsilon \alpha^2\}
 \end{equation}
 and let $\overline{I} = \{1,\ldots,N\} \setminus I$. Since $2/p > 1$,
 \begin{align}
 &\left(1+\sum_{n=1}^{N} \frac{\alpha^p}{\beta^p}\Big(\big(\frac{z_{n}}{\alpha^{2}} + 1\big)^{p/2}  - 1\Big)\right)^{2/p}\nonumber\\
&\quad\ge \left(1+\sum_{n\in I} \frac{\alpha^p}{\beta^p}\Big(\big(\frac{z_{n}}{\alpha^{2}} + 1\big)^{p/2}  - 1\Big)\right)^{2/p}\nonumber\\
&\quad\;\;\;+ \left(\sum_{n\in \overline{I}} \frac{\alpha^p}{\beta^p}\Big(\big(\frac{z_{n}}{\alpha^{2}} + 1\big)^{p/2}  - 1\Big)\right)^{2/p}\nonumber\\
&\quad\ge 1+\sum_{n\in I} \frac{p\alpha^{p-2}}{2\beta^p} \Big(1-\frac{2-p}{4}\epsilon\Big)z_{n}\nonumber\\
&\quad\;\;\;+ \left(\sum_{n\in \overline{I}} \frac{(\epsilon+1)^{p/2}-1}{\epsilon^{p/2}\beta^p} z_{n}^{p/2}\right)^{2/p}.
\end{align}
If
\begin{align}
&\frac{p\alpha^{p-2}}{2\beta^p} \Big(1-\frac{2-p}{4}\epsilon\Big) \eta^2\ge 1\label{e:compcondglob1}\\
&\frac{\big((\epsilon+1)^{p/2}-1\big)^{2/p}}{\epsilon\beta^2} \eta^{2}\ge 1, \label{e:compcondglob2}
\end{align}
then
 \begin{multline}\label{e:globprevtrinqpre}
 \left(1+\sum_{n=1}^{N} \frac{\alpha^p}{\beta^p}\Big(\big(\frac{z_{n}}{\alpha^{2}} + 1\big)^{p/2}  - 1\Big)\right)^{2/p}\\
\ge 1 + \sum_{n\in I} \frac{z_n}{\eta^{2}} +\left(\sum_{n\in \overline{I}} \frac{z_n^{p/2}}{\eta^{p}} \right)^{2/p}.
\end{multline}
In addition, as $p/2< 1\le q/2$,
\begin{align}
&\Big(1 + \sum_{n\in I} \frac{z_n}{\eta^{2}}\Big)^{q/2} \ge 1 + \sum_{n\in I} \frac{z_n^{q/2}}{\eta^q}\\
& \left(\sum_{n\in \overline{I}} \frac{z_n^{p/2}}{\eta^p} \right)^{2/p} \ge \left(\sum_{n\in \overline{I}} \frac{z_n^{q/2}}{\eta^q} \right)^{2/q},
\end{align}
where the last inequality follows from \eqref{eq_lp-lq}.
This yields
 \begin{multline}\label{e:globprevtrinq}
 \left(1+\sum_{n=1}^{N} \frac{\alpha^p}{\beta^p}\Big(\big(\frac{z_{n}}{\alpha^{2}} + 1\big)^{p/2}  - 1\Big)\right)^{2/p}\\
\ge \left( 1 + \sum_{n\in I} \frac{z_n^{q/2}}{\eta^q} \right)^{2/q}+\left(\sum_{n\in \overline{I}} \frac{z_n^{q/2}}{\eta^q} \right)^{2/q}.
\end{multline}
We deduce Inequality \eqref{e:ineqglobopt} by applying the triangle inequality for the $\ell^{q/2}$ norm to the right-hand side of \eqref{e:globprevtrinq}.
The provided condition in \eqref{e:simpcondglob} corresponds to the choice $\epsilon = 1$ in \eqref{e:compcondglob1}-\eqref{e:compcondglob2}.

\section{Proof of Proposition \ref{proposPsi}}
\label{Proof1}
(i) Let us first show that $\Psi$ is Lipschitz-differentiable\ldr{by investigating}{. To do so, we will investigate} the properties of $\nabla^2 \Psi$.
We start by studying the behavior of $|||\nabla^2 \Psi_1(\ex)|||$, where $\ex\in \mathbb{R}^N$ and the spectral norm is denoted by $|||.|||$.
Using \eqref{e:D2ellpx} and \eqref{e:D2phi1}, we obtain
\begin{align}
\allx{}
|||\nabla^2 \Psi_1(\ex)||| & \leq   \frac{ |||\DD{(\mathbf{Z})}|||}{ \ell_{p,\alpha}^p(\ex)+\beta^p} +\frac{p \|\mathbf{Y}\|^2}{\left( \ell_{p,\alpha}^p(\ex)+\beta^p \right)^2} 
\end{align}
where we make use of the shorter notation:
\begin{equation}
\begin{cases}
\mathbf{Y} = \left( x_n (x_n^2 + \alpha^2)^{\frac{p}{2}-1}\right)_{1\leq n \leq N}\\
\mathbf{Z} = \left(\left((p-1)x_n^2+\alpha^2\right) (x_n^2+\alpha^2)^{\frac{p}{2}-2}\right)_{1\le n \le N}
\end{cases}
\end{equation}
First, we have 
\begin{eqnarray}
\lefteqn{\frac{ |||\DD (\mathbf{Z})|||}{\ell_{p,\alpha}^p(\ex)+\beta^p}} \nonumber  \\
& = & \frac{1}{ \ell_{p,\alpha}^p(\ex)+\beta^p} \nonumber \\
& &  \times \sup_{1\le n \le N} \left(\left|(p-1)x_n^2+\alpha^2 \right| (x_n^2+\alpha^2)^{\frac{p}{2}-2}\right).
\end{eqnarray}
Since $p<2$ and, for every $n\in \{1,\ldots,N\}$,
\begin{equation}
|(p-1)x_n^2+\alpha^2|\le x_{n}^{2}+\alpha^2,
\end{equation}
we deduce that
\begin{eqnarray}
\lefteqn{\frac{ |||\DD (\mathbf{Z})|||}{\ell_{p,\alpha}^p(\ex)+\beta^p}} \nonumber  \\
& \leq & \frac{1}{ \beta^p}  \sup_{1\le n \le N} \Big((p-1)(x_n^2+\alpha^2)^{\frac{p}{2}-1} + \alpha^2(x_n^2+\alpha^2)^{\frac{p}{2}-2}\Big)
\nonumber\\ 
& \leq & p\frac{\alpha^{p-2}}{\beta^p}.
\label{eq:Lip1Psi1}
\end{eqnarray}
Besides, by setting $\nu = \sum_{n=1}^N \left(x_n^2 + \alpha^2\right)^{p/2}$,
\begin{align}
\frac{\|\mathbf{Y}\|^2}{\left( \ell_{p,\alpha}^p(\ex)+\beta^p \right)^2} &= \frac{1}{\left( \ell_{p,\alpha}^p(\ex)+\beta^p \right)^2} \sum_{n=1}^N x_n^2 (x_n^2+\alpha^2)^{p-2}\nonumber\\
 &\le \frac{1}{\left( \ell_{p,\alpha}^p(\ex)+\beta^p \right)^2} \sum_{n=1}^N \frac{x_n^2}{(x_n^2+\alpha^2)^{2}}(x_n^2+\alpha^2)^{p}\nonumber\\
 &\le \frac{1}{2\alpha^{2}\left( \ell_{p,\alpha}^p(\ex)+\beta^p \right)^2} \sum_{n=1}^N (x_n^2+\alpha^2)^{p}\nonumber\\
&\le  \frac{1}{2\alpha^2\left( \ell_{p,\alpha}^p(\ex)+\beta^p \right)^2}  \Big(\sum_{n=1}^N  (x_n^2+\alpha^2)^{p/2}\Big)^2 \nonumber\\
&= \frac{\nu^2}{2\alpha^2(\nu-N\alpha^p+\beta^p)^2} \nonumber\\
& =  \frac{1}{2\alpha^2} \Big(1+\frac{N\alpha^p-\beta^p}{\nu-N\alpha^p+\beta^p}\Big)^{2}\nonumber\\
& \le \frac{1}{2\alpha^2} \max\Big\{1,\Big(\frac{N\alpha^p}{\beta^p}\Big)^{2}\Big\}.
\label{eq:Lip2Psi1}
\end{align}
These results prove that $\nabla^2 \Psi_1$ is bounded

Let us now study the Hessian of $\Psi_2$ at $\ex \in \eR^N$. Let $\epsilon \in ]0,+\infty[$, let
\begin{equation}\label{e:defLambda}
\boldsymbol{\Lambda}_\epsilon = \frac{\nabla^2 \ell_{q,\eta}^q(\ex) + q(q-1)\epsilon  \boldsymbol{I}_N}{\ell_{q,\eta}^q(\ex)},
\end{equation}
and let
\begin{equation}
\nabla^2_\epsilon \Psi_2(\ex)  = \frac{1}{q}\left( \boldsymbol{\Lambda}_\epsilon - \frac{ \nabla \ell_{q,\eta}^q(\ex) \left(\nabla \ell_{q,\eta}^q(\ex) \right)^\top}{\ell_{q,\eta}^{2q}(\ex)}  \right)
\end{equation}
By continuity,
\begin{equation}
\label{e:contD2phi2eps}
\lim_{\epsilon \to 0} |||\nabla^2_\epsilon \Psi_2(\ex)||| = |||\nabla^2 \Psi_2(\ex)|||
\end{equation}
On the other hand, since $\boldsymbol{\Lambda}_\epsilon$ is a positive definite matrix, 
\begin{equation}
\nabla^2_\epsilon \Psi_2(\ex) = \frac{1}{q}  \boldsymbol{\Lambda}_\epsilon^{1/2} (\boldsymbol{I}_N - \mathbf{v}_\epsilon \mathbf{v}_\epsilon^\top) \boldsymbol{\Lambda}_\epsilon^{1/2}
\end{equation}
where, by using \eqref{e:D1ellqx},
\begin{align}
\label{eq:vepsilon}
&\mathbf{v}_{\epsilon}\nonumber\\
& = \boldsymbol{\Lambda}_{\epsilon}^{-1/2} \frac{\nabla \ell_{q,\eta}^q(\ex)}{\ell_{q,\eta}^{q}(\ex)} \nonumber\\
& = \sqrt{\frac{q}{q-1}} \frac{1}{\ell_{q,\eta}^{q/2}(\ex)}\Big[\frac{\operatorname{sign}(x_1)|x_1|^{q-1}}{\sqrt{|x_1|^{q-2}+\epsilon}},\ldots,\frac{\sign(x_N)|x_N|^{q-1}}{\sqrt{|x_N|^{q-2}+\epsilon}}\Big]^\top.
\end{align}
Therefore,
\begin{align}\label{e:normnabla2Psi2eps}
|||\nabla^2_\epsilon \Psi_2(\ex)||| &=||| \frac{1}{q}  \boldsymbol{\Lambda}_\epsilon^{1/2} (\boldsymbol{I}_N - \mathbf{v}_\epsilon \mathbf{v}_\epsilon^\top) \boldsymbol{\Lambda}_\epsilon^{1/2} |||\nonumber\\
& \le \frac{1}{q} |||  \boldsymbol{\Lambda}_\epsilon |||\, ||| \boldsymbol{I}_N - \mathbf{v}_\epsilon \mathbf{v}_\epsilon^\top |||.
\end{align}
According to \eqref{e:defLambda},
\begin{align}
\boldsymbol{\Lambda}_\epsilon =
 &  \frac{q(q-1)}{\ell_{q,\eta}^q(\ex)}\,\DD{\left((|x_n|^{q-2}+\epsilon)_{1 \leq n \leq N}\right)}.
\end{align}
Consequently,
\begin{align}
||| \boldsymbol{\Lambda}_\epsilon ||| &=  \frac{q(q-1)}{\ell_{q,\eta}^{q}(\ex)}  \, \sup_{1\le n \le N} (|x_n|^{q-2}+\epsilon).
\end{align}
We thus derive from \eqref{e:normnabla2Psi2eps} that
\begin{align}
|||\nabla^2_\epsilon \Psi_2(\ex)||| 
& = \frac{q-1}{\ell_{q,\eta}^{q}(\ex)} \sup_{1\le n \le N} (|x_n|^{q-2}+\epsilon)\nonumber\\
& \quad  \quad \times \max\{1, \|\mathbf{v}_\epsilon\|^2-1\}
\end{align}
As $\epsilon\to 0$, \eqref{e:contD2phi2eps} yields
\begin{equation}
|||\nabla^2 \Psi_2(\ex)||| \le \frac{q-1}{\ell_{q,\eta}^{q}(\ex)} \sup_{1\le n \le N} |x_n|^{q-2} \;\max\{1, \|\mathbf{v}\|^2-1\}
\end{equation}
where, according to \eqref{eq:vepsilon},
\begin{align}
\|\mathbf{v}\|^2 = \lim_{\epsilon \to 0} \|\mathbf{v}_\epsilon\|^2 = \frac{q}{q-1} \frac{\sum_{n=1}^N |x_n|^q}{\ell_{q,\eta}^{q}(\ex)}
\end{align}
which is equivalent to
\begin{align}
\|\mathbf{v}\|^2 -1 
&=\frac{1}{q-1} \Big(1-\frac{q \eta^q}{\ell_{q,\eta}^{q}(\ex)}\Big) 
\end{align}
Since $\Big(1-\frac{q \eta^q}{\ell_{q,\eta}^{q}(\ex)}\Big) < 1$ and $\frac{1}{q-1} < 1$ for all $q \in [2, +\infty[$, we deduce that
$\|\mathbf{v}\|^2 -1  < 1.$
Resultingly,
\begin{align}
|||\nabla^2 \Psi_2(\ex)||| &\le \frac{q-1}{\ell_{q,\eta}^{q}(\ex)} \sup_{1\le n \le N} |x_n|^{q-2} \nonumber\\
&= \frac{q-1}{(\eta^q+\sum_{n=1}^N |x_n|^q)^{2/q}} \left(\frac{\sup_{1\le n \le N} |x_n|^{q}}{\eta^q+\sum_{n=1}^N |x_n|^q}\right)^{\frac{q-2}{q}}\nonumber\\
& \leq \frac{q-1}{(\eta^q+\sum_{n=1}^N |x_n|^q)^{2/q}} \label{e:majnormD2phi2pre} \\
& \leq \frac{q-1}{\eta^2} \label{e:majnormD2phi2}.
\end{align}
\ldr{F}{It can be concluded f}rom the boundedness of $\nabla^2 \Psi_1$ and $\nabla^2 \Psi_2$\ldr{, }{that} $\nabla^2 \Psi = \nabla^2 \Psi_1 - \nabla^2 \Psi_2$ is bounded, hence $\Psi$ is a Lipschitz-differentiable and
\begin{multline}
||| \nabla^2 \Psi(\ex) ||| = ||| \nabla^2 \Psi_1(\ex) - \nabla^2 \Psi_2(\ex) ||| 
\\ \leq ||| \nabla^2 \Psi_1(\ex) ||| + ||| \nabla^2 \Psi_2(\ex) |||.
\end{multline}
Using \eqref{eq:Lip1Psi1}, \eqref{eq:Lip2Psi1}, and \eqref{e:majnormD2phi2}, we conclude that
\begin{equation}
||| \nabla^2 \Psi(\ex) ||| \leq p \frac{\alpha^{p-2}}{\beta^p} +  \frac{p}{2\alpha^2} \max\Big\{1,\Big(\frac{N\alpha^p}{\beta^p}\Big)^{2}\Big\} + \frac{q-1}{\eta^2},\nonumber
\end{equation}
hence $\Psi$ is Lipschitz differentiable with constant $L$ as in \eqref{eq:LipSPOQ}.

(ii) Let us now prove that $\Psi$ satisfies the majorization inequality \eqref{e:majloc}. By noticing that  $\xi \mapsto (\xi+\alpha^2)^{p/2}$ is a concave function, it follows from standard majorization properties \cite{Hunter_D_2004_j-american-statistician_tutorial_mma} that for every 
$(\ex',ex) \in (\eR^N)^2$, and $ n \in \{1,\ldots,N\}$:
\begin{align}
(x_n'^2+\alpha^2)^{p/2} \leq &  (x_n^2+\alpha^2)^{p/2} + p x_n (x_n^2+\alpha^2)^{p/2-1} (x_n'-x_n) \nonumber \\
 & + \frac{p}{2} (x_n^2+\alpha^2)^{p/2-1} (x'_n-x_n)^2.
\end{align}
As a consequence,
\begin{equation}
\ell_{p,\alpha}^p(\ex') \le \ell_{p,\alpha}^p(\ex)+ (\ex'-\ex)^\top\nabla  \ell_{p,\alpha}^p(\ex)+\frac{p}{2} (\ex'-\ex)^\top \eA_1(\ex) (\ex'-\ex)\nonumber
\end{equation}
where
$\eA_1(\ex) = \DD \left( (x_n^2+\alpha^2)^{p/2-1}\big)_{1\le n \le N} \right)$. 
By using the Napier inequality expressed as
\begin{equation}
\big(\forall (u,v) \in ]0,+\infty[^2) \qquad  \log u \le \log v + \frac{u-v}{v},
\end{equation}
we get
\begin{align}
\Psi_1(\ex') \leq & \Psi_1(\ex)+(\ex'-\ex)^\top \nabla\Psi_1(\ex) \nonumber \\
& + \frac{1}{2(\ell_{p,\alpha}^p(\ex)+\beta^p)}(\ex'-\ex)^\top \eA_1(\ex) (\ex'-\ex).
\label{e:quadmajphi1}
\end{align}
By applying the descent lemma to function $-\Psi_2$, and using \eqref{e:majnormD2phi2pre} we obtain
\begin{align}
\big(\forall (\ex,\ex')\in \overline{\BB}_{q,\rho}^2\big) \quad -\Psi_2(\ex') & \leq -\Psi_2(\ex) - (\ex'-\ex)^\top \nabla\Psi_2(\ex) \nonumber \\
& + \frac{\chi_{q,\rho}}{2} \|\ex'-\ex\|^2.
\label{e:quadmajphi2}
\end{align}
The majorization property is then derived from \eqref{e:quadmajphi1} and \eqref{e:quadmajphi2}. 
\eca{The inequality \eqref{eq:maj_bnd} can be deduced in a straightforward manner, by noticing that both $\ell_{p,\alpha}^p(\ex)$ and $(x_n^2+\alpha^2)^{p/2-1}$ are minimal for $\ex = \zerob_N$.} 



\section*{Acknowledgment}
\ldr{The authors thank Prof. Audrey Repetti (Heriot-Watt University, UK) for initial motivations and thoughtful discussions, and Prof. Marc-André Delsuc (University of Strasbourg, France), for helping with MS  problem modeling.}{The authors thank Prof. Audrey Repetti from Heriot-Watt University, UK, for providing initial motivation and thoughtful discussions on this work, \afa{and Prof. Marc-André Delsuc from the University of Strasbourg, France, for helping to modelize the MS problem.}}
\ifCLASSOPTIONcaptionsoff
  \newpage
\fi




\tiny
\bibliographystyle{IEEEtran}
\bibliography{abbr,Reference_Base_JabRef}

\end{document}